\documentclass[11pt,reqno]{amsart}

\usepackage{setspace}
\usepackage{color}
\usepackage[all]{xy}
\usepackage{amssymb, amsmath, amsfonts, amsthm}
\usepackage{enumitem}

\calclayout

\usepackage{amsmath,amssymb}
\usepackage{amsthm}
\usepackage{mathabx}
\usepackage{amsfonts}
\usepackage{pst-node}
\usepackage{graphicx}
 \usepackage{pifont}
 \usepackage[T1]{fontenc}

\newtheorem{theorem}{Theorem}[section]
\newtheorem{proposition}[theorem]{Proposition}
\newtheorem{corollary}[theorem]{Corollary}
\newtheorem{lemma}[theorem]{Lemma}
\theoremstyle{definition}
\newtheorem{definition}[theorem]{Definition}
\newtheorem{example}[theorem]{Example}
\newtheorem{conjecture}[theorem]{Conjecture}
\newtheorem{remark}[theorem]{Remark}

\makeatletter
\@addtoreset{equation}{section}
\makeatother

\newcommand{\bbZ}{{\mathbb Z}} 

\newcommand{\bbF}{{\mathbb F}} 

\newcommand{\bG}{{\mathbf G}}

\newcommand{\cA}{{\mathcal A}}
\newcommand{\cB}{{\mathcal B}}
\newcommand{\cC}{{\mathcal C}}
\newcommand{\cD}{{\mathcal D}}
\newcommand{\cE}{{\mathcal E}}
\newcommand{\cF}{{\mathcal F}}

\newcommand{\cH}{{\mathcal H}}

\newcommand{\cK}{{\mathcal K}}
\newcommand{\cL}{{\mathcal L}}

\newcommand{\cN}{{\mathcal N}}

\newcommand{\cO}{{\mathcal O}}

\newcommand{\cT}{{\mathcal T}}

\newcommand{\id}{\mathrm{id}}

\DeclareMathOperator*{\colim}{colim}

\def\maprt#1{\smash{\,\mathop{\longrightarrow}\limits^{#1}\,}}

\newcommand{\Aut}{{\rm Aut}}
\newcommand{\Inn}{{\rm Inn}}

\newcommand{\Hom}{\mathrm{Hom}}
\newcommand{\Coind}{\mathrm{Coind}}
\newcommand{\Ind}{\mathrm{Ind}}
\newcommand{\Mor}{\mathrm{Mor}}

\newcommand{\Ext}{\mathrm{Ext}}
\newcommand{\Tor}{\mathrm{Tor}}

\newcommand{\Res}{\mathrm{Res}}
\newcommand{\Ob}{\mathrm{Ob}}
\newcommand{\Ini}{\mathrm{Ini}}
\newcommand{\Fin}{\mathrm{Fin}}

\begin{document}
\title{LHS-spectral sequences for regular extensions of categories}
 
\author{Erg\"un Yal{\c c}{\i}n}
\address{Department of Mathematics, Bilkent University, 06800 
Bilkent, Ankara, Turkey}

\email{yalcine@fen.bilkent.edu.tr}

\begin{abstract}  In [F. Xu, On the cohomology rings of small categories, J. Pure Appl. Algebra 212 (2008), 2555-2569], a LHS-spectral sequence for target regular extensions of small categories is constructed.
We extend this construction to ext-groups and 
construct a similar spectral sequence for source regular extensions 
(with right module coefficients). As a special case of these LHS-spectral sequences, we obtain
three different versions of S\l omi\' nska's
spectral sequence for the cohomology of regular EI-categories. 
We show that many well-known spectral sequences 
related to the homology decompositions of finite groups, centric linking systems, and
the orbit category of fusion systems can be obtained as the LHS-spectral sequence of an extension.  
\end{abstract}



\maketitle


\setcounter{tocdepth}{1}
\tableofcontents

\section{Introduction}

A functor $\pi: \cC \to \cD$ between two small categories is called \emph{target regular} if the following conditions hold:
\begin{enumerate}
\item  $\Ob (\cC)=\Ob (\cD)$ and $\pi$ is the identity map on objects,
\item For each $x, y\in \Ob (\cC)$, the group $K(y):=\ker \{ \pi _{y,y}: \Aut_{\cC} (y) \to \Aut _{\cD} (y) \}$
acts freely on $\Mor _{\cC} (x, y)$, and
\item the map $\pi_{x,y}: \Mor _{\cC} (x,y) \to \Mor _{\cD} (x, y)$ 
induces a bijection  $$K(y) \backslash \Mor _{\cC} (x, y) \cong \Mor _{\cD} (x, y).$$
\end{enumerate}
 If $\pi :\cC \to \cD$ is a target regular functor with kernel $\cK :=\{K(x)\}_{x\in \Ob (\cC)}$, then we write $$\cE: \cK \maprt{i} \cC \maprt{\pi} \cD $$ where $\cK$ is considered as a category (a disjoint union of group categories) and $i$ is the functor which is identity on objects and such that $i_x : K(x) \to \Aut _{\cC} (x)$ is the inclusion map for every $x\in \cC$. We call such a sequence of functors a \emph{target regular extension of categories}.

We say that the functor $\pi : \cC \to \cD$ is \emph{source regular} if the opposite functor $\pi ^{op} : \cC^{op} \to \cD ^{op}$ is target regular. The kernel $\cK$ is defined in a similar way and the extension $\cE: \cK \maprt{i} \cC \maprt{\pi} \cD$ is called a source regular extension.
 
There are many examples of source and target regular extensions of categories that appear naturally. Let $G$ be a discrete group and $\cH$ be a collection of subgroups of $G$ (closed under conjugation). The \emph{orbit category} of $G$ over the collection $\cH$ is the category $\cO _{\cH} (G)$ whose objects are subgroups $H \in \cH$, and whose morphisms $H \to K$ are given by $G$-maps $G/H \to G/K$. We can identify the set of $G$-maps $G/H \to G/K$ with the set of right cosets 
$\{ Kg \, | \, g\in G, gHg^{-1} \leq K\}$.  The \emph{transporter category} is the category $\cT_{\cH} (G)$ whose objects are subgroups $H \in \cH$, and whose morphisms $H \to K$ are given by $N_G (H, K)=\{ g \in G \, |\, gHg^{-1} \leq K\}$. The quotient map $N_G(H, K)\to K \backslash N_G (H, K)$ defines a target regular extension of categories 
\begin{equation}\label{eqn:Intro-SubExt}
\cE:  \{ K\} _{K \in \cH} \maprt{i} \cT _{\cH } (G) \maprt{\pi}  \cO _{\cH} (G).
\end{equation}

The \emph{fusion category} of a discrete group $G$ over a collection $\cH$ is the category $\cF _{\cH } (G)$ whose objects are subgroups $H \in \cH$, and whose morphisms $H \to K$ are given by conjugation maps $c_g: H\to K$ with $g\in G$, defined by $c_g (h) = ghg^{-1} $ for every $h\in H$.
There is a source regular extension 
\begin{equation}\label{eqn:Intro-CentExt}
\cE: \{ C_G (H) \} _{H \in \cH} \maprt{i} \cT _{\cH } (G) \maprt{\pi}  \cF _{\cH} (G)
\end{equation}
induced by the quotient map $N_G (H, K) \to N_G (H, K)/ C_G (H)\cong\Mor _{\cF _{\cH } (G)} (H, K)$. 
There are also many other examples of target and source regular extensions related to fusion orbit category 
and linking systems. We introduce them in Sections \ref{sect:Transporter} and \ref{sect:Linking}.

In this paper we construct spectral sequences for target and source regular extensions of categories similar to  
Lyndon–Hochschild–Serre spectral sequences for group extensions. 
The cohomology version of a LHS-spectral sequence for target regular extensions was constructed by Xu \cite{Xu-CohSmall} (see also 
Oliver and Ventura \cite[Prop A.11]{OliverVentura}). One can also find a construction of LHS-spectral sequences for homology groups
of pre-cofibered functors in Quillen's lecture notes in \cite[p. 139]{Penner-Book}.  
We extend these constructions to ext-groups and construct a similar spectral sequence 
for source regular extensions with right module coefficients.  
 
Let $R$ be a commutative ring with unity. A contravariant functor $M: \cC ^{op} \to R$-Mod is called 
a \emph{right $R\cC$-module}. A functor $M: \cC \to R$-Mod is called a \emph{left $R\cC$-module}. 
Since we are mostly interested in right $R\cC$-modules, and since it is good to fix a choice, 
\emph{when we say $M$ is an $R\cC$-module we always mean it is a right $R\cC$-module.} 
The cohomology of small category $\cC$ with coefficients in an $R\cC$-module $M$ is defined by
$$H^* (\cC; M) :=\Ext ^* _{R\cC} (\underline R, M),$$ where for each $n\geq 0$, $\Ext^n_{R\cC} (\underline R, M)$ denotes the $n$-th 
ext-group over the category of $R\cC$-modules and 
$\underline R$ denotes the constant module (see Section \ref{sect:Cohomology} for details). 
 
\begin{theorem}[Xu \cite{Xu-CohSmall}]\label{thm:Intro-SSXu} Let $\cE : \cK \maprt{i} \cC \maprt{\pi} \cD $ be a target regular extension. Then
for every $R\cC$-module $M$, there is a spectral sequence
$$ E_2 ^{p,q}= H^p ( \cD;  H^q ( \cK ; M) ) \Rightarrow H^{p+q} ( \cC ;  M)$$
where $H^q (\cK; M)$ denotes the $R\cD$-module that sends $x\in \Ob (\cC)$ to the cohomology group
$H^q (K(x); \Res_{i_x} M)$.
\end{theorem}

We review Xu's construction, provide more details about the module structure of $H^q (\cK; M)$, and show that there is an ext-group version of the LHS-spectral sequence for extensions of categories.

\begin{theorem}\label{thm:Intro-SSXuExt} Let $\cE : \cK \maprt{i} \cC \maprt{\pi} \cD $ be a target regular extension.
Then for every $R\cD$-module $N$ and $R\cC$-module $M$, there is a spectral sequence
$$ E_2 ^{p,q}= \Ext ^p _{R\cD} (N,  H^q ( \cK ; M) ) \Rightarrow \Ext ^{p+q} _{R \cC} ( \Res_{\pi} N, M)$$
where $H^q (\cK; M)$ denotes the $R\cD$-module that sends $x\in \Ob (\cC)$ to the cohomology group
$H^q (K(x), \Res_{i_x} M)$. 
\end{theorem}

If $\pi : \cC \to \cD$ is a source regular functor, then $\pi ^{op} : \cC ^{op} \to \cD ^{op} $ is a target regular functor. 
So the spectral sequence above gives a spectral sequence for source regular extensions with coefficients 
in left $R\cC$-module $M$ and a left $R\cD$-module $N$ (see Corollary \ref{cor:XuSSExt-Source}). 
However if $\pi : \cC \to \cD$ is source regular functor and if we want to relate the cohomology 
of $\cC$ and $\cD$ with right module coefficients, then 
the spectral sequence above can not be used. For this case, we construct the following 
spectral sequence involving the homology groups of the groups in the kernel $\cK=\{ K(x) \}$. 

\begin{theorem}\label{thm:Intro-SSGYGen}
Let $\cE : \cK \maprt{i} \cC \maprt{\pi}\cD $ be a source regular extension. For every 
$R\cC$-module $N$ and $R\cD$-module $M$,  
there is a spectral sequence
$$ E_2 ^{p,q}= \Ext ^p _{R \cD} ( H_q ( \cK; N), M) \Rightarrow \Ext ^{p+q} _{R\cC} (N,  \Res _{\pi} M)$$
where $H_q ( \cK; N)$ denotes the (right) $R\cD$-module that sends $x \in \Ob(\cC)$ to the homology group 
$H_q (K(x); \Res_{i_x} N)$.
\end{theorem}

We prove Theorems \ref{thm:Intro-SSXuExt} and \ref{thm:Intro-SSGYGen}  by studying the derived functors of induction and coinduction functors. Associated to a functor $F: \cC\to \cD$, there is a restriction functor $$\Res_F: R\cD\text{-Mod} \to R\cC\text{-Mod}$$ defined by precomposition with $F$. The induction functor and coinduction functors $$\Ind _F, \Coind _F : R\cC\text{-Mod} \to R\cD\text{-Mod}$$ in the other direction
are defined as the left and right adjoints of the restriction functor.  
The derived functors of induction and coinduction functors can be calculated by 
using the cohomology and homology of certain comma categories (see 
Propositions \ref{pro:RightDerFunc} and \ref{pro:LeftDerFunc}). 
Using these calculations, we obtain two versions of the Gabriel-Zisman spectral sequence (see Theorems \ref{thm:GZ} and \ref{thm:VarGZ}).
These spectral sequences, with some modification, gives the spectral sequences in Theorems \ref{thm:Intro-SSXuExt} and \ref{thm:Intro-SSGYGen}.

In Section \ref{sect:RegularEI} we apply these results to regular EI-categories. A small category $\cC$ is called an EI-category if every endomorphism in $\cC$ is an isomorphism. The set of isomorphism classes $[\cC]$ of an EI-category $\cC$ is a poset with the order relation given by
$[x]\leq [y] $ if $\Mor _{\cC} (x, y)$ is nonempty. Consider the functor $\cC\to [\cC]$ that sends each $x\in \Ob (\cC)$ to its isomorphism class $[x]$.
We say that the EI-category $\cC$ is target (resp. source) regular if the functor $\cC\to [\cC]$ is target (resp. source) regular.
This gives three different versions of LHS-spectral sequences for regular EI-categories (see Theorem \ref{thm:RegularSS}).  

The subdivision category $S(\cC)$ of an EI-category $\cC$ is the category whose objects are the chains of composable morphisms in $\cC$:
$$ \sigma :=(  \sigma_0 \xrightarrow{\alpha_1} \sigma_1 \xrightarrow{\alpha_2} \cdots  \xrightarrow{} \sigma_{n-1} \xrightarrow{\alpha_n} \sigma_n)$$
where each $\alpha_i$ is a non-isomorphism. The morphisms in $S(\cC)$ are defined by inclusion of subchains via a sequence of isomorphisms (see Definition \ref{def:Subdivision}). The subdivision category $S(\cC)$ is a source regular category, hence its opposite category $s(\cC)$ is target regular. For every $R\cC$-module $M$, there is an isomorphism $$H^* (s(\cC); \Res_{\Ini} M) \cong H^* (\cC; M)$$ where $\Ini: s(\cC) \to \cC$ is the functor that sends a chain $\sigma \in s(\cC)$ to the first object $\sigma_0 \in \cC$ in the chain (see Propositions \ref{pro:SC-Cohomology} and \ref{pro:SC-SourceRegular}). 

For each $i$, let  $\pi_i : \Aut _{s(\cC)} (\sigma) \to \Aut _{\cC} (\sigma _i)$ denote the group homomorphism defined by
restricting an automorphism of $\sigma \in s(\cC)$ to the automorphism of the $i$-th object in the chain.
Applying the spectral sequence in Theorem \ref{thm:RegularSS} to the category $s(\cC)$, and using the isomorphism above, 
we obtain the following spectral sequence. 

\begin{theorem}[{S\l omi\' nska \cite{Slominska-Homotopy}}]\label{thm:Intro-SlomSS}
Let $\cC$ be an EI-category and $M$ be an $R\cC$-module. Then, for each $q \geq 0$, the assignment 
$[\sigma] \to H^q ( \Aut_{s(\cC)} (\sigma); \Res _{\pi_0} M( \sigma_ 0 ))$ defines an $R[s(\cC)]$-module, denoted by $\cA ^q$, and 
there is a spectral sequence
$$ E_2 ^{p,q}= H^p ( [s(\cC)] ;  \cA^q) \Rightarrow H^{p+q}  ( \cC ;  M).$$ 
\end{theorem}

This spectral sequence is a reformulation of a spectral sequence due to S\l omi\' nska \cite[Cor 1.13]{Slominska-Homotopy}.  The trivial coefficient version of this spectral sequence was constructed by Linckelmann \cite[Thm 3.1]{Linckelmann-OnH}
and used for studying the gluing problem in block theory. 

In Section \ref{sect:Transporter}, we apply the spectral sequences above to the extensions involving the transporter category $\cT_{\cH} (G)$. The LHS-spectral sequences for the extensions in \ref{eqn:Intro-SubExt} and \ref{eqn:Intro-CentExt} give spectral sequences
that can be considered as the subgroup and centralizer decompositions for the transporter category (see Propositions \ref{pro:SubgroupDec} and \ref{pro:CentralizerDec}).  Applying the spectral sequence in Theorem \ref{thm:Intro-SlomSS} to the transporter category $\cT_{\cH} (G)$
gives a spectral sequence that can be considered as the normalizer decomposition for the transporter category 
(see Proposition \ref{pro:NormalizerDec}). When we take the coefficients to be the constant functor $\bbF_p$ and $\cH$ to be an ample collection,
these spectral sequences for the transporter category give spectral sequences for mod-$p$ homology decompositions of the classifying space $BG$.

In Section \ref{sect:Linking}, we consider $p$-local finite groups $(S, \cF, \cL)$ and obtain spectral sequences for the mod-$p$ 
homology decompositions of the centric linking system $\cL$ (with nontrivial coefficients). In particular we show that  there is a target 
regular extension of categories 
$$ \cE:  \{ P\} _{P\in \cF^c} \xrightarrow{j} \cL \xrightarrow{\widetilde \pi} \cO ^c (\cF) $$ 
which gives a spectral sequence
that converges to the cohomology of $\cL$ with nontrivial coefficients in the following form:

\begin{theorem}\label{thm:IntroSubgroupSS}
Let $(S, \cF, \cL)$ be a $p$-local finite group. For every $R\cL$-module $M$, there is 
a spectral sequence
$$ E_2 ^{p,q}= H^p (\cO ^c (\cF)  ;  \cH^q_M )  \Rightarrow H^{p+q} (\cL ; M)$$ 
where $\cH^q _M$ is the $R\cO ^c (\cF)$-module such that  for every $P\in \cF^c$, $\cH^q_M (P)=H^q (P; \Res _{j_P} M (P))$.  
\end{theorem}

This is proved as Proposition \ref{pro:SubgroupSS} in the paper. If we apply Theorem  \ref{thm:Intro-SlomSS} to the linking system $\cL$ and to the orbit category $\cO (\cF^c)$, we obtain two spectral sequences. These spectral sequences use the cohomology of the 
category  $\overline{s}d (\cF^c)$ which is the category  of  $\cF$-conjugacy classes of chains in $\cF^c$ (see  Definition \ref{def:CatConjClasses}). 
 
\begin{theorem}\label{thm:Intro-NormDecLinking}
Let $(S, \cF, \cL)$ be a $p$-local finite group. For every $R\cL$-module $M$,
there is a spectral sequence 
$$E_2 ^{p.q} = H^p (\overline{s}d (\cF^c) ; \cA_{\cL} ^q) \Rightarrow H^{p+q} (\cL ; M)$$ 
where $\cA_{\cL} ^q$ is the $R\overline{s}d (\cF^c)$-module such that $\cA_{\cL} ^q ([\sigma])= H^q (\Aut _{\cL } (\sigma); \Res _{\pi_0} M(P_0))$
for every $[\sigma]=[P_0<\cdots<P_n]$ in $\overline{s}d(\cF^c)$.
\end{theorem}

This spectral sequence can also be constructed as the Bousfield-Kan spectral sequence
for the normalizer decomposition for $|\cL|$ introduced by Libman \cite{Libman-NormDec}. 
Finally we apply Theorem \ref{thm:Intro-SlomSS} to the fusion orbit category $\cO (\cF^c)$
and obtain a spectral sequence of the following form:

\begin{theorem}\label{thm:IntroOrbitFusion}
Let $(S, \cF, \cL)$ be a $p$-local finite group. Then for every right $R\cO (\cF^c)$-module $M$, there is a spectral sequence
$$E_2 ^{p.q} = H ^p ( \overline{s}d (\cF^c) ; \cA_{\cO} ^q) 
\Rightarrow H^{p+q} (\cO (\cF ^c) ; M)$$ 
where $\cA_{\cO} ^q$ is the left $R\overline{s}d (\cF^c)$-module such that
 $\cA_{\cO} ^q ([\sigma])= H^q (\Aut _{s(\cO (\cF^c))} (\sigma); \Res _{\pi_0} M(P_0))$. 
\end{theorem} 

This spectral sequence has some formal similarity to the spectral sequence 
constructed in \cite[Thm 1.7]{Yalcin-Sharp} for studying the sharpness problem. 
However, the 
$E_2$-terms of the spectral sequence obtained here
is in general different than the one constructed in \cite[Thm 1.7]{Yalcin-Sharp}.

\emph{Notation:} Throughout the paper, when we say $\cH$ is a collection, we always assume that 
$\cH$ is closed under conjugation. $R$ always denotes an associative commutative ring with unity. 
We say $R$ is a $p$-local ring if all the primes other than $p$ are invertible in $R$.  
When we say $M$ is an $R\cC$-module without further specifications, we always mean that $M$ is a right $R\cC$-module.
We denote the cohomology of the category $\cC$
with coefficients in a left $R\cC$-module $M$ and in a right $R\cC$-module $M$ both by $H^* (\cC; M)$. This is the notation used by L\" uck in \cite{Lueck-Book}  and we follow this convention throughout the paper. 
 
 {\bf Acknowledgements:}  We thank the anonymous  referee for the careful reading of the paper and for many corrections and valuable suggestions. 
 In particular, some of the references in the current version were provided by the referee. We also thank the editor 
 for the smooth handling of the paper.


\section{Modules over small categories}\label{sect:Modules} 

In this preliminary section we introduce necessary definitions for doing homological
algebra over small categories. In the second part of the section we introduce the induction 
and coinduction functors
associated to a functor $F: \cC \to \cD$.   
 For more details on this material, we refer the reader to 
\cite{Lueck-Book}, \cite{Richter-Book}, \cite{Schubert-Book}, \cite{KashiwaraSchapira-Book},
\cite{Webb-Survey}, and \cite{HPY}.

\subsection{$R\cC$-modules} Let $\cC$ be a small category. We denote the set of objects in $\cC$ by $\Ob (\cC)$, and the set of morphisms from
$x$ to $y$ by $\Mor_{\cC} (x,y)$ or simply by $\cC(x,y)$. 
For each $x\in \Ob (\cC)$, the group of automorphisms of $x$ in $\cC$ is denoted by $\Aut_{\cC} (x)$.
If we reverse all the arrows in $\cC$, the category we obtain is called the opposite category of $\cC$, denoted by $\cC^{op}$.

\begin{definition} Let $R$ be a commutative ring with unity. A contravariant functor  $M:\cC^{op} \to R$-Mod
is called a \emph{right $R\cC$-module}. A covariant functor $M: \cC \to R\text{-Mod}$
is called a \emph{left $R\cC$-module}.
 \emph{When we say $M$ is an $R\cC$-module, it will always mean that it is a right $R\cC$-module, unless otherwise stated explicitly.} \end{definition}

Let $M$ be an $R\cC$-module. For each object $x\in \Ob (\cC)$, the $R$-module $M(x)$ can be considered as a right
$R \Aut _{\cC} (x)$-module where the action of $a \in \Aut (x)$ is given by the $R$-module homomorphism $M(a) : M(x) \to M(x)$. 
An $R\cC$-module homomorphism $f: M \to N$ is a natural transformation of functors. 
The category of $R\cC$-modules is an abelian category and the usual 
notions of kernel, cokernel, and exact sequence exist and they are defined objectwise (see \cite[\S 1.6]{Weibel-Book}).
For example, a short exact sequence of $R\cC$-modules $M_1\to M_2 \to M_3$ is exact if for every $x\in \Ob(\cC)$, the sequence 
$$M_1 (x)  \to M_2 (x) \to M_3 (x)$$ is exact. The kernel $\ker f$ of a homomorphism $f: M \to N$ is the $R \cC$-module such that for every
$x \in \Ob (\cC)$, we have $(\ker f) (x) =\ker \{ f(x) : M(x) \to N(x) \}$. 
In the category of $R\cC$-modules, there are enough projectives. Projective modules are constructed using representable functors as follows:
 
\begin{definition} For each $x\in \Ob (\cC)$, let $R\cC(?, x)$ denote the $R\cC$-module that sends  every $y\in \Ob (\cC)$ to the free $R$-module $R\cC (y, x)$.
For each morphism $\varphi : y\to z$, the induced map $R\cC(z, x) \to R\cC (y,x)$ is defined by precomposition with $\varphi$.
\end{definition}

Using the Yoneda lemma, one can prove the following:

\begin{lemma}\label{lem:YonedaRight}
For every $x\in \Ob(\cC)$ and every $R\cC$-module $M$, there is an isomorphism of $R$-modules
$$\Hom _{R \cC } (R\cC (? , x) , M)  \cong M(x).$$ 
Consequently, for every $x \in \Ob (\cC)$, the $R\cC$-module $R\cC (? , x) $ is a projective module.
\end{lemma}

\begin{proof} See \cite[Prop 4.4]{Webb-Survey}.
\end{proof}

Using Lemma \ref{lem:YonedaRight}, one can also show that for each $R\cC$-module $M$, there is a projective 
$R\cC$-module $P$ and surjective map $P \to M$ (see \cite[\S 9.16, \S 9.19]{Lueck-Book}).
This shows that the category of $R\cC$-modules has enough projectives.
There are also enough injectives in the category of $R\cC$-modules (see \cite[p. 43]{Weibel-Book}).


\subsection{Restriction and induction functors}

\begin{definition}\label{def:Restriction} Given a functor $F: \cC \to \cD$ between two small categories, the \emph{restriction functor} $$\Res_F : R\cD\text{-Mod} \to R\cC\text{-Mod}$$ is defined by precomposing with $F$, i.e. for an $R\cD$-module $M$, 
$\Res_F M:=M \circ F^{op}$.
\end{definition}

In the other direction there is an induction functor $\Ind _F : R\cC\text{-Mod} \to R\cD\text{-Mod}$
which is defined to be the functor that is left adjoint to the restriction functor $\Res_F(-)$. 
To show the existence of the left adjoint, one defines the induction 
functor either by using left Kan extensions or by using tensor products. We explain both approaches.
 
\begin{definition}\label{def:IndKanExt} 
For every $R\cC$-module $M$, we define  $\Ind _F (M)$ 
to be the left Kan extension $LK_{F^{op}} (M)$ of the functor $M: \cC ^{op} \to R$-Mod 
via the functor $F ^{op} : \cC ^{op}  \to \cD ^{op} $. 
\end{definition} 

For this definition to make sense, 
one needs to show that the left Kan extension  $LK_{F^{op}} (M)$ exists for every $R\cC$-module $M$. 
For this we need to recall the explicit description of Kan extensions 
using colimits over comma categories.  

\begin{definition}\label{def:Comma}
Let $F: \cC\to \cD$ be a functor between two small categories. For every $d\in \Ob (\cD)$, the comma category $d \backslash F$ is the category whose objects are the pairs $(c, f)$ where $c\in \Ob (\cC)$ and $f \in \Mor _{\cD} ( d , F(c) )$, and whose morphisms $(c, f) \to (c', f')$ are given by the morphisms $\varphi : c\to c'$ in $\cC$ such that $f= f' \circ F(\varphi)$. The comma category $F /d$ is defined in a similar way with objects $(c, f)$ where $c\in \Ob (\cC)$ and $f\in \Mor _{\cD} (F(c), d)$.
\end{definition}

If $F : \cC\to \cD$ and $M: \cC \to \cE$ are functors where $\cE$ is a cocomplete category (i.e., where the colimits exists), then the left Kan extension $LK_F (M)$ exists and for every $d\in \Ob(\cD)$, $$LK_F (M) (d) \cong  \colim _{(c, f) \in F/d} (M \circ \pi_{\cC})$$
where $\pi_{\cC} : F/d \to \cC$ is the functor that sends $(c, f) \in \Ob (F/d)$ to $c \in \Ob (\cC)$ (see \cite[Thm 4.1.4]{Richter-Book}, 
 \cite[Thm 2.3.3]{KashiwaraSchapira-Book}).  Since the category of $R$-modules is cocomplete (see \cite[Prop 8.4.3]{Schubert-Book}), the left Kan extension $LK _{F^{op} } (M)$ exists for any $R\cC$-module $M$.  Explicit description of left Kan extensions gives the following formula for induction functor.

\begin{lemma}\label{lem:IndKanFormula} Let $F: \cC \to \cD$ be a functor and $M$ be a $R\cC$-module. Then for every $d\in \cD$, 
$$(\Ind_F M)(d) \cong \colim _{(c, f) \in F^{op} /d} (M \circ \pi_{\cC}) \cong \colim _{ (c, f) \in (d\backslash F) ^{op}} (M \circ \pi_{\cC})$$
where in the last expression $\pi_{\cC}$ is the functor $(d \backslash F)^{op} \to \cC ^{op}$ that 
sends the object $(c, f)$ in $d \backslash F$ to $c \in \Ob(\cC)$.
\end{lemma}

Since by definition the left Kan extension is the left adjoint of the restriction functor (see  \cite[Def 2.3.1, Thm 2.3.3]{KashiwaraSchapira-Book}), 
we have  the following:

\begin{proposition}\label{pro:IndLeftAdjoint} Let $F: \cC\to \cD$ be a functor.
For every $R\cC$-module $M$ and $R\cD$-module $N$, there is a natural 
isomorphism $$\Hom _{R\cD} (\Ind _F M, N ) \cong \Hom _{R \cC} (M, \Res_F N).$$
\end{proposition}

As a consequence of the adjointness of the induction and restriction functors, we have:

\begin{corollary}\label{cor:ResProj} For every functor $F: \cC \to \cD$, the induction functor $\Ind _F : R\cC\text{-Mod} \to R\cD\text{-Mod}$ takes projectives to projectives.
 \end{corollary}

\begin{proof}  Since the restriction functor $\Res_F(-)$ preserves exact sequences, its left adjoint, the induction functor 
$\Ind _F (-)$, takes projectives to projectives. 
\end{proof}

Since taking colimits is not an exact functor, the induction functor $\Ind _F (-)$ is not exact in general. 
In Section \ref{sect:DerivedInduction}, we calculate its left derived functors in terms of homology groups of the comma category 
$d\backslash F$.


\subsection{Tensor products over $R\cC$}

We recall the definition of the tensor product of two modules over $R\cC$.
 
\begin{definition}[{\cite[\S 9.12]{Lueck-Book}}]\label{def:Tensor}
Let $M$ be a right $R\cC$-module and $N$ be a left $R\cC$-module. The \emph{tensor product}
of $M$ and $N$ over $R\cC$ is the $R$-module 
$$M \otimes _{R\cC} N := \Bigl ( \bigoplus _{x \in \Ob (\cC) }  M(x) \otimes N (x)  \Bigr ) /J$$
where $J$ is the ideal generated by the elements of the form $$M(\varphi) (m) \otimes n - m\otimes N(\varphi) (n)$$
over all $\varphi \in \cC(x, y)$, $m\in M(y)$, and $n\in N(x)$.
\end{definition} 
 
The induction functor $\Ind _F(-)$ associated to a functor $F: \cC \to \cD$ is defined using the tensor 
product as follows:

\begin{definition}\label{def:Induction}   For a functor $F: \cC \to \cD$, let  $R\cD(??, F(?))$ denote the 
functor $\cD^{op} \times \cC \to R\text{-Mod}$ that sends $(d , c) \in \Ob (\cD) \times \Ob (\cC)$ 
to the $R$-module $R\cD (d , F(c))$.  We can consider $R\cD (??, F(?))$ as an $R\cC$-$R\cD$-bimodule where the left $R\cC$-module structure and the  right $R\cD$-module structure are defined by restricting $R\cD(??, F(?))$ to the corresponding module categories.
For every $R\cC$-module $M$, the $R\cD$-module $\Ind _F M$ is defined by $$\Ind _F M :=M \otimes _{R\cC}  R\cD(??, F(?)).$$
\end{definition}

The definition of tensor product in Definition \ref{def:Tensor} gives the following 
explicit description for $(\Ind _F M) (d)$ for every $d \in \Ob (\cD)$:
 $$(\Ind _F M) (d) = \Bigl ( \bigoplus _{c \in \Ob (\cC) } 
M(c) \otimes_R R \cD (d, F(c)) \Bigr ) /J$$
where $J$ is the ideal generated by the elements of the form $M(\varphi) (m) \otimes f - m \otimes (F(\varphi) \circ f)$ over all  $\varphi \in \cC (c, c')$,
$m \in M(c')$, and  $f \in \cD (d , F(c))$. 

\begin{example} Let $G$ be a group and $H$ be a subgroup of $G$. If  $F:\mathbf{H} \to \mathbf{G}$ is the functor 
between two group categories induced the inclusion map, then the induction functor $\Ind _F(-)$ coincides with
the usual induction $\Ind_H ^G (-)$ defined by $\Ind _H ^G M = M \otimes _{RH} RG$.
\end{example}

The induction functor $\Ind _F (-)$  defined using tensor products is left adjoint to the restriction functor 
(see \cite[\S 9.22]{Lueck-Book}).
By the uniqueness of the left adjoints up to unique isomorphism (see \cite[Cor 16.4.4]{Schubert-Book}), we can conclude that
the two definitions of induction using Kan extensions and using tensor products 
coincide. One can see this also by observing that the formula for $\Ind (M) (d)$ in both definitions give the same module.

We end this section with a well-known lemma on tensor product of modules over categories. 
We will use this lemma in our proofs in the paper.

\begin{lemma}\label{lem:TensorIsom} Let $M$ be a left $R\cC$-module and $N$ be a right 
$R\cC$-module. Then for every 
$z \in \Ob(\cC)$,
there are natural isomorphism of $R$-modules:
\begin{enumerate}
\item[(i)] $R\cC (?, z) \otimes _{R\cC} M \cong M (z).$
\item[(ii)] $N \otimes _{R\cC}  R\cC(z, ?) \cong N (z).$
\end{enumerate}
\end{lemma}

\begin{proof} We prove the first one, the argument for the second one is similar. 
By the definition of tensor product over $R \cC$, for every $z\in \Ob (\cC)$, we have 
$$R\cC (?, z) \otimes _{R\cC} M = \Bigl (\bigoplus _{x \in \Ob(\cC) } R\cC (x, z) \otimes _R M(x)  \Bigr )/J$$
where $J$ is the ideal generated by $f \otimes M(\varphi) (m) -  (f \circ \varphi) \otimes m$ over all 
$\varphi \in \cC(x, y)$, $m\in M(x)$, and $f \in \cC (y, z)$.
Note that in the quotient module, an element $[\tau \otimes m]$, with $\tau \in \cC (x, z)$ and $m \in M(x)$,
is equal to $[\id _z \otimes M(\tau) (m)]$, and two elements $[\id _z  \otimes m]$ and $[\id _z \otimes m']$ with $m, m' \in M(z)$ 
being equal only if $m=m'$. Hence, $R \cC (? , z) \otimes _{R\cC} M \cong M(z)$. The naturality is clear from this argument.
\end{proof}

A proof of Lemma \ref{lem:TensorIsom} can also be given using category algebras since $R\cC(?, z) =1_z \cdot R\cC$ and
$R\cC (z, ?) = R\cC\cdot 1_z$ (see \cite{Webb-Survey} for details).

\subsection{Coinduction functor}
Let $F: \cC \to \cD$ be a functor between two small categories. Associated to $F$, 
there is a coinduction functor $$\Coind _F : R\cC\text{-Mod} \to R\cD\text{-Mod}$$ 
defined as the functor which is right adjoint to the restriction functor. As in the case 
of the induction functor, the coinduction functor can be defined  by using the limit 
description of Kan extensions. There is also a definition of it using hom-functor and bimodules.

\begin{definition}[{\cite[\S 9.15]{Lueck-Book}}]\label{def:BimoduleCoInd}
Let $R \cD (F(?), ??)$ denote the functor
$$\cC ^{op} \times \cD \to R\text{-Mod}$$ that sends a pair $(c, d) \in \Ob (\cC) \times \Ob (\cD)$ 
to the $R$-module $R\cD (F(c), d)$. Note that for each $d\in \Ob (\cD)$, $R\cD (F (?) , d)$ is a
right $R\cC$-module, and for each $c\in \Ob(\cC)$, $R\cD (F(c), ??)$ is a left $R\cD$-module.
Because of this two sided module structure, we say $R\cD (F(?), ??)$ is an
$R\cD$-$R\cC$-bimodule.
\end{definition}

The coinduction functor is defined as follows:

\begin{definition}\label{def:Coind}  
Let $F: \cC \to \cD$ be a functor between two small categories. 
For an $R\cC$-module $M$, $\Coind _F M$ is defined 
to be the $R\cD$-module such that 
$$(\Coind _F M) (??)= \Hom _{R\cC} ( R \cD( F(?), ?? ) , M).$$
\end{definition}
 
We have the following adjointness property:
 
\begin{proposition}[{\cite[\S 9.22]{Lueck-Book}}]\label{pro:CoindRightAdjoint} 
For every $R\cD$-module $N$ and every $R\cC$-module $M$, there is 
a natural isomorphism $$\Hom _{R\cD} (N, \Coind _F M) \cong \Hom _{R \cC} (\Res_F N, M).$$
Moreover, the  coinduction functor $\Coind_F : R\cC\text{-Mod} \to R\cD\text{-Mod}$ 
takes injectives to injectives. 
\end{proposition}

\begin{proof} 
The first sentence follows from the adjointness of the tensor product and the Hom functor (see \cite[Thm 15.1.3]{Richter-Book}), 
The second sentence follows from the adjointness property and from the fact that the restriction
functor is exact. 
\end{proof}

As in the case of the induction functor, the coinduction functor can be defined using Kan extensions.  The right Kan extension
$RK_{F} (M)$ of  the functor $M: \cC \to \cE$ along a functor $F: \cC\to \cD$ is defined to be the right adjoint of the restriction 
functor $\Res _F(-)$. The right Kan extension $RK_F (M)$ exists if the category $\cE$ is complete (if limits in $\cE$ exist) and
we have 
$$RK_{F} (M) (d) \cong \lim _{(c, f) \in d \backslash F} (M \circ \pi_{\cC})$$
where $\pi_{\cC} : d \backslash F \to \cC$ is the functor that sends $(c,f: d \to F(c))$ to $c$ in $\Ob(\cC)$ (see \cite[Def 2.3.1, Thm 2.3.3]{KashiwaraSchapira-Book}).

Since the category of $R$-modules is complete, the right Kan extension $RK_{F^{op}} (M)$ exists 
for every functor $F: \cC \to \cD$ and $R\cC$-module $M$. By the uniqueness of right adjoints, these two definitions coincide.
We have the following formula for the coinduction functor.
 
\begin{lemma}\label{lem:CoindFormula} Let $F: \cC \to \cD$ be a functor and $M$ be an $R\cD$-module. Then, for every $d \in \Ob (\cD)$, 
$$(\Coind_F M) (d) \cong \lim _{(c, f) \in d \backslash F^{op}  } (M \circ \pi_{\cC}) \cong \lim _{(c,f)  \in (F/d) ^{op} } (M\circ \pi _{\cC})$$
where $\pi_{\cC} : (F/d)^{op} \to \cC^{op}$ in the last expression is the functor that sends $(c, f) \in \Ob(F/d)$ to $c \in \Ob(\cC)$.   
\end{lemma}
 
 
\section{Cohomology of small categories}\label{sect:Cohomology}

In this section we introduce the basic definitions for cohomology of small categories. 
The definition of cohomology of a small category with coefficients in a natural system 
was introduced by Baues and Wirsching in \cite{BauesWirsching}. The definition of 
cohomology of small categories with coefficients in right and left $R\cC$-modules 
was introduced earlier (see, for example, \cite{Quillen-KTheory}). More details 
on this material can be found in \cite{BauesWirsching}, 
\cite[Chp 16]{Richter-Book}, \cite{Webb-Survey}, and \cite{Huseinov}.

\subsection{Ext-groups and cohomology of a small category} Let $N$ be an $R\cC$-module. 
The assignment $M \to \Hom_{R\cC} (N, M)$ defines a covariant functor $\Hom _{R\cC} (N, -)$ 
from the abelian category  of $R\cC$-modules to $R$-modules.
This functor is left exact. For each $n\geq 0$, we denote its $n$-th right derived functor 
by $R^n \Hom_{R\cC} (N, -)$.

\begin{definition}\label{def:ExtGroup} Let $M, N$ be $R\cC$-modules. For every $n \geq 0$, 
the \emph{ext-group of $N$ and $M$} is defined by
$$ \Ext ^n _{R\cC} (N, M) :=[R^n \Hom _{R\cC} (N, -)](M).$$ 
\end{definition}

By the balancing theorem \cite[Thm 2.7.6]{Weibel-Book}, the ext-group $\Ext _{R\cC} ^n (N, M)$ can also 
be calculated as the  $n$-th cohomology of the cochain complex $\Hom _{R\cC} (P_* , M)$ where 
$P_*$ is an $R\cC$-projective resolution of $N$.

\begin{definition}\label{def:CohCat} Let $\cC$ be a small category. The \emph{constant functor} $\underline R$ over 
$\cC$ is the $R\cC$-module 
that sends every object $x\in \Ob (\cC)$ to $R$, and every morphism to the  identity map $\id _R: R \to R$. 
For every $R\cC$-module $M$, for $n \geq 0$, the \emph{$n$-th cohomology group of $\cC$}
with coefficients in $M$ is defined by $$ H^n (\cC ; M) :=\Ext  ^n _{R \cC} (\underline R, M).$$ 
\end{definition}

The cohomology group $H^n (\cC; M)$ can also be viewed as the $n$-th derived functor of the limit functor over $\cC$.
The category of $R\cC$-modules is complete and the limit of an $R\cC$-module $M$ over $\cC$ is defined as follows:
$$ \lim _{\cC} M := \{ (m_x) \in \prod _{x \in \cC} M(x) \, | \,  M(\alpha) (m_y)=m_x \text{ for every }\alpha : x\to y \text{ in } \cC\}. $$
The limit functor $\lim _{\cC} (-)$ is a left exact functor and the $n$-th right derived functor of $\lim _{\cC} (-)$ is called 
the \emph{$n$-th higher limit of $M$}, and it is denoted by $\lim _{\cC} ^n M$. From the definitions it follows that:

\begin{proposition} For every $R\cC$-module $M$, there is an isomorphism $$\Hom_{R\cC}  (\underline R, M)\cong \lim _{\cC} M.$$ 
Hence for every $n \geq 0$, there is an isomorphism $\underset{\cC\ }{\lim{}^n} M \cong H^n (\cC; M).$
\end{proposition}

If $\cC$ is equal to the group category $\bG$ for a group $G$, then an $R\bG$-module $M$ is a module over the group ring $RG$, and
the cohomology of $\cC$ with coefficients in $M$ coincides with the cohomology of the group $G$ with coefficients in the $RG$-module $M$ 
(see \cite{Brown-Book} for the definition of cohomology of a group).


\subsection{Homology of a small category}

Homology of a category is defined using tor-groups
$\Tor _* ^{R\cC} (M, N)$. The tor-groups over a category is called the functor homology and 
it is a well studied area. We refer the reader to 
 \cite[Chp 15]{Richter-Book}, \cite{Vespa-Survey}, \cite{DjamentTouze-Functor}, \cite{Huseinov}
 for more details on this material.

\begin{definition}\label{def:HigherTor} 
Let $M$ be a left $R\cC$-module. 
The functor $- \otimes _{R\cC} M : N \to N \otimes _{R\cC} M$ which sends a right $R\cC$-module $N$ to the $R$-module $N \otimes _{R\cC} M$
is right exact. For a right $R\cC$-module $N$ and for every $n \geq 0$,  the $n$-th tor-group is defined by
$$\Tor _n ^{R\cC} (N, M):= L_n (-\otimes _{R\cC} M) (N)$$ where
$L_n (-\otimes _{R\cC} M)$ is the $n$-th left derived functor of the functor $-\otimes _{R\cC} M$.
\end{definition}

By the definition of left derived functors, tor-groups can be calculated as follows: 
Let $N$ be a right $R\cC$-module and $P_* \to N$ be a right $R\cC$-projective resolution of 
$N$. Then for every left $R\cC$-module $M$, the tor-group 
$\Tor _n ^{R\cC} (N, M)$ is the $n$-th homology group of the chain complex
$P_* \otimes _{R\cC} M$. Tor-groups can also be computed using a double complex.

\begin{theorem}[{Weibel, \cite[Thm 2.7.2]{Weibel-Book}}]\label{thm:BalancingTor} 
Let $N$ be a right $R\cC$-module and $M$ be a left $R\cC$-module.
Then $$\Tor_n ^{R\cC} (N, M):=L_n(N \otimes _{R\cC} -) (M) \cong L_n (-\otimes _{R\cC} M) (N)\cong 
H_n (Tot(P_* \otimes _{R\cC} Q_*))$$
where $P_* \to N$ is a right $R\cC$-projective resolution of 
$N$ and $Q_* \to M$ is a left $R\cC$-projective resolution of $M$. 
\end{theorem}

The homology groups of the category $\cC$ are defined as follows:

\begin{definition}\label{def:HomCat}
Let $\cC$ be a small category. \emph{The $n$-th homology group of $\cC$} with coefficient 
in a right $R\cC$-module $N$ is defined by $H_n (\cC, N) := \Tor ^{R\cC} _n (N, \underline{R})$. For a left $R\cC$-module $M$, the 
$n$-th homology group of $\cC$ with coefficients in $M$ is defined by $H_n (\cC, M) := \Tor ^{R\cC} _n (\underline{R}, M)$. 
\end{definition}

 
\subsection{Classifying space of a category} For a small category $\cC$, the nerve $\cN \cC$ is the simplicial 
set whose $n$-simplices are given by an $n$-fold 
sequence of composable morphisms 
$$\sigma:=(x_0 \maprt{\alpha_1} x_1 \maprt{\alpha_2} x_2 \maprt{\alpha_3} \cdots \maprt{\alpha_n} x_n)$$
and the face maps $d_i: (\cN \cC)_n \to (\cN \cC)_{n-1}$ are defined by deleting the $i$-th object in the chain
and the degeneracy maps $s_i : (\cN \cC) _n \to (\cN \cC)_{n+1}$ are 
defined by adding the identity map $\id_{x_i} : x_i \to x_i$  to the chain
(see \cite[Section III.2.2]{AKO-Book}  
for details). 

The classifying space $B\cC$ of $\cC$ is defined to be the geometric realization of the nerve $\cN \cC$.
The topological space $B\cC$ is a CW-complex, and for any abelian group $A$, we denote its cellular (and singular) cohomology 
by $H^* (B\cC; A)$. The cohomology group $H^* (B\cC; A)$ is isomorphic to the cohomology of $\cC$ with coefficients
in the constant functor $\underline A$ (see \cite[16.2.3]{Richter-Book}).

If $F: \cC \to \cD$ is a functor, then it defines a simplicial map between corresponding nerves and it gives 
a continuous map $BF: B\cC\to B\cD$ between the classifying spaces. We have the following:

\begin{proposition}\label{pro:Homotopic} Let $F, F': \cC \to \cD$ be two functors. If there is a natural 
transformation $\mu: F \to F'$, then the induced maps $BF, BF' : B \cC \to B\cD$ are homotopic.
\end{proposition}

\begin{proof} See \cite[Prop III.2.1]{AKO-Book}.
\end{proof}

As a consequence, we obtain that if  $\cC$ and $\cD$ are equivalent categories
then $B\cC$ and $B\cD$ are homotopy equivalent, hence $H^* (B\cC; A)\cong H^* (B\cD; A)$ for every abelian group $A$.
In particular, if the category $\cC$ has an initial object, then $B\cC$ is contractible and $H^* (B\cC; A)\cong H^* (pt; A)$.

\section{Functoriality and Shapiro's lemma}

\subsection{Functoriality of the cohomology of a category}

A special case of the cohomology of a small category is the cohomology of a group. In that case, for every group 
homomorphism $f: G \to H$, there is an induced homomorphism  $f^*: H^* (H ; M) \to H^* (G; \Res_f M)$. 
For the cohomology of small categories, there is a similar induced homomorphism. 

\begin{proposition}\label{pro:Functorial} Let $F: \cC \to \cD$ be a functor between two small categories. 
Then for every $R\cD$-module $M$, the functor $F$
induces a graded $R$-module homomorphism $$F^* : H^* (\cD ; M) \to H^* (\cC; \Res _F M).$$ 
\end{proposition}

\begin{proof} This is well-known, but for completeness we add a proof here.  
Let $\underline R_{\cC}$ and $\underline R_{\cD}$ denote the constant functors 
for the categories $\cC$ and $\cD$, respectively. There is an isomorphism 
$j: \underline R_{\cC} \to \Res _F \underline R_{\cD}$ such that for each $x\in \Ob(\cC)$, 
$j(x): \underline R _{\cC} (x)\to \underline R_{\cD} (F(x))$ is the identity map $\id_R: R \to R$. Let 
$\mu: \Ind _F \underline R _{\cC} \to \underline R_{\cD}$ the $R\cD$-module homomorphism 
associated to the map $j$ under the isomorphism in Proposition \ref{pro:IndLeftAdjoint}. 

Let $P_*$ be an $R\cC$-projective resolution of $\underline R_{\cC}$, and $Q_*$ be an $R\cD$-projective 
resolution of $\underline R _{\cD}$. Consider the chain complex $\Ind _F P_* \to \Ind _F \underline R_{\cC}$ 
obtained by applying the functor $\Ind _F (-)$ to the complex $P_* \to \underline R$. By Corollary \ref{cor:ResProj},
$\Ind _F P_*$ is a chain complex of projective $R\cD$-modules. Hence there is a chain map 
$\mu_* :  \Ind _F P_* \to Q_*$ covering  the homomorphism
$\mu : \Ind _F \underline R _{\cC}  \to \underline R_{\cD}$. For every $R\cD$-module $M$, this 
induces a homomorphism 
$$H^* (\cD; M) \cong H^* ( \Hom _{R\cD } (Q_* , M) ) \xrightarrow{\mu^*} H^* ( \Hom _{R\cD} ( \Ind _F P_* , M)).$$
By Definition \ref{def:CohCat} and Proposition \ref{pro:IndLeftAdjoint}, we have
$$H^* ( \Hom _{R\cD} (\Ind_F P_*, M)) \cong  H^* (\Hom _{R\cC} (P_* , \Res _F M))  \cong H^* (\cC; \Res_F M).$$
Combining these, we obtain the homomorphism $F^* : H^* (\cD ; M) \to H^* (\cC; \Res _F M).$
\end{proof}

More generally, in \cite{Webb-Bisets} a notion of biset functor for categories is defined and it is shown that 
for a field $R$, $H^* (\cC; R)$ is  a biset functor on $\mathbb{B} _R ^{\text{all}, 1}$ (see \cite[Thm 6.7]{Webb-Bisets}).

\subsection{Functoriality of the homology of a category}

For a functor $F:\cC \to \cD$, the induction functor for a left $R\cC$-module is defined in a similar way to the definition of
the induction functor for a right $R\cC$-module. Let $R\cD(F(?), ??)$ denote 
the $R\cD$-$R\cC$-bimodule defined in Definition \ref{def:BimoduleCoInd}.
For a left $R\cC$-module $M$, we have
$$\Ind _F M := R\cD(F(?), ??)\otimes _{R\cC} M.$$

We have the following well-known fact for the induction functor.

\begin{lemma}\label{lem:TensorAdjoint} 
Let $F:\cC\to \cD$ be a functor. Then for every right $R\cD$-module $N$ and $R\cC$-module $M$, there is
an isomorphism of $R$-modules $$ N\otimes _{R\cD} (\Ind _F M) \cong (\Res _F N) \otimes _{R\cC} M.$$
\end{lemma}

\begin{proof} This follows easily from definitions and from Lemma \ref{lem:TensorIsom}.
\end{proof}

We now show that homology of a category is functorial.

\begin{proposition}[{\cite[Cor 3.11]{Vespa-Survey}}]\label{pro:FunctorialHom}
Let $F: \cC\to \cD$ be a functor. For every $R\cD$-module $N$, there is a
graded $R$-module homomorphism
$$F_* : H_* ( \cC; \Res _F N ) \to H_* (\cD; N).$$ 
\end{proposition}

\begin{proof}  
Let  $\mu: \Ind _F \underline R _{\cC} \to \underline R_{\cD}$ denote the $R\cD$-module homomorphism between 
corresponding constant functors (as left modules).
Let $P_*$ be a left $R\cC$-projective resolution 
of $\underline R_{\cC}$, and $Q_*$ be a left
$R\cD$-projective resolution of $\underline R _{\cD}$.
There is a chain map $\mu_*: \Ind _F P_* \to Q_*$ covering $\mu$. 
By Definition \ref{def:HomCat} and
Lemma \ref{lem:TensorAdjoint} , we have
$$H_* (\cC; \Res _F N) \cong H_* (\Res _F N \otimes _{R\cC} P_*) \cong
H_* ( N \otimes _{R\cD} \Ind _F P_*).$$
The chain map $\mu_*$ induces a graded $R$-module homomorphism 
$$H_* ( N  \otimes _{R\cD} \Ind _F P_*  )  \xrightarrow{\mu_*}  H_* (N \otimes _{R\cD}  Q_* ) \cong H_* (\cD; N).$$
Combining these, we obtain the homomorphism $F_* : H_* (\cC ; \Res _F N) \to H^* (\cD ; N).$ 
\end{proof}


\subsection{Shapiro's isomorphism}\label{sect:Shapiro}

Let $H$ be a subgroup of a group $G$, and $F : \mathbf{H} \to \mathbf{G}$ be the functor defined by the inclusion map
$H \to G$. Then we have
$\Ind _F M \cong M \otimes _{RH} RG$. In this case $\Ind _H ^G (-)$ is an exact functor
since $RG$ is free as an $RH$-module.  There are other examples of functors $F:\cC\to \cD$ where 
the associated induction functor $\Ind _F (-)$ is exact. An interesting example of a such functor
is the inclusion functor between two orbit categories $\cO_{\cH} (H) \to \cO _{\cH} (G)$ where $H$ is a subgroup of $G$
(see \cite[Lem 2.9]{HPY}). A similar exactness result is proved in 
\cite[Prop 4.5]{Yalcin-Sharp} for functors $\overline \cF _{\cH} (H) \to \overline \cF_{\cH} (G)$ between two fusion orbit 
categories. 

When the induction functor is exact, then there is an isomorphism of ext-groups, called \emph{Shapiro's isomorphism}:
 
\begin{proposition}\label{pro:IndShapiro} Let $F:\cC \to \cD$ be a functor between two small categories. 
Assume that the induction functor $\Ind _F : R\cC\text{-Mod} \to R\cD\text{-Mod}$ is exact. Then 
for every $R\cC$-module $N$ and $R\cD$-module $M$, there is an isomorphism 
$$\Ext^* _{R\cD} (\Ind _F N, M) \cong \Ext ^*_{R\cC} (N, \Res_F M).$$
\end{proposition}

\begin{proof} This follows easily from the adjointness of induction and restriction functors.
\end{proof}

One special case where the induction functor $\Ind_F (-)$ is exact is the case where $F$ has a right adjoint. 
We explain this case in detail in Section \ref{sect:AdjointPair}. 
 
There is also a version of Shapiro's lemma for coinduction functor: 

\begin{proposition}\label{pro:CoindShapiro} Let $F:\cC \to \cD$ be a functor between two small categories. 
Assume that the coinduction functor $\Coind _F (-)$ is exact. Then 
for every $R\cD$-module $N$ and every $R\cC$-module $M$, there is a natural isomorphism 
$$\Ext^* _{R\cD} (N, \Coind _F M) \cong \Ext ^*_{R\cC} (\Res _F N, M).$$
\end{proposition}

\begin{proof} This follows from the adjointness of coinduction and restriction functors.
\end{proof}

If we take $N=\underline{R}$ in Proposition \ref{pro:CoindShapiro}, we obtain 
Shapiro's isomorphism for the cohomology of categories. 

\begin{corollary}[{\cite[Lem 3.1]{JM-Elementary}, \cite[Lem 3.5]{LeviRagnarsson}}]\label{cor:CoindShapiro} 
Let $F:\cC \to \cD$ be a functor between two small categories. 
Assume that the coinduction functor $\Coind _F : R\cC\text{-Mod} \to R\cD\text{-Mod}$ is exact. Then 
for every $R\cC$-module $M$, there is an isomorphism 
$$H^* (\cD; \Coind _F M) \cong H^* (\cC ; M).$$
\end{corollary}

 
 \section{Gabriel-Zisman Spectral Sequence}\label{sect:GabrielZisman}
  
In this section, we first review the construction of the Gabriel-Zisman spectral sequence 
for cohomology of categories (see \cite[Appendix II, Thm 3.6]{GZ-Book}).  Then we construct 
a version of the Gabriel-Zisman spectral sequence 
involving the derived functors of the induction functor.  

 \subsection{Derived functors of the coinduction functor}
Let $F : \cC \to \cD$ be a functor and $M$ be an $R\cC$-module. For each $d\in \Ob (\cD)$, 
let $F/d$ denote the comma category and  $\pi _{\cC} : (F/d) ^{op} \to \cC^{op}$ is the functor 
that sends $(c, f)\in F/d$ to $c \in \Ob (\cC)$. For each $n \geq 0$, let $H^n (F/d; M \circ \pi _{\cC} )$ 
denote the $n$-th cohomology
of the comma category $F/d$ with coefficients in the (right) $R(F/d)$-module
$M \circ \pi _{\cC}$.
 By Proposition \ref{pro:Functorial}, for every morphism $\varphi : d \to d'$ in $\cD$, there is an induced $R$-module
homomorphism  
$$(F/\varphi)^* : H^n (F/d' ; M \circ \pi_{\cC} ^{d'} ) \to H^n (F/d; M \circ \pi _{\cC} ^d ). $$

\begin{lemma}\label{lem:RD-ModuleStructure}
The assignment $d \to H^n (F/d ; M \circ \pi  _{\cC} )$, together with the induced 
homomorphisms $(F/\varphi) ^*$ for every $\varphi: d \to d'$
defines an $R\cD$-module. We denote this $R\cD$-module by $H^n (F/-; M \circ \pi_{\cC} )$.
\end{lemma}

\begin{proof} This is clear from the definitions.
\end{proof}

We have the following:

\begin{proposition}\label{pro:RightDerFunc} Let $F : \cC \to \cD$ be a functor and $M$ be an $R\cC$-module. 
Then the coinduction functor $Coind_F(-)$ is left exact  and its $n$-th right derived functor calculated 
at $M$, $(R^n \Coind _F) (M)$, is isomorphic to the $R\cD$-module
$H^n (F/- ; M \circ \pi_{\cC})$. 
\end{proposition}

\begin{proof} By Lemma \ref{lem:CoindFormula}, for every $d \in \Ob (\cD)$, 
$$(\Coind_F M) (d) \cong \lim _{(c,f)  \in (F/d) ^{op} } (M\circ \pi _{\cC}).$$
This gives that 
$$(R^n \Coind _F) (M) (d) \cong \underset{(F/d)^{op}}{\lim {}^n}  \, (M\circ \pi _{\cC} ) \cong H^n (F/d ; M \circ \pi_{\cC})$$
for every $d \in \Ob (\cD)$. Since these isomorphisms are natural, this gives the desired result.
\end{proof}


\subsection{Gabriel-Zisman spectral sequence}

We first recall the Grothendieck spectral sequence defined for a composition of functors.

\begin{theorem}[{\cite[Thm 5.8.3]{Weibel-Book}}]\label{thm:Grothendieck}
Let $\cA$, $\cB$, and $\cC$ be abelian categories such that both $\cA$ and $\cB$ have enough injectives.
Suppose that $G: \cA \to  \cB$ and $F : \cB \to \cC$ are left exact functors such that  $G$ sends injective objects of $\cA$
to $F$-acyclic objects of $\cB$ ($B \in \Ob(\cB)$ is $F$-acyclic if $R^i F (B)=0$ for $i\neq 0$). Then for every $A\in \cA$, there exists a convergent first quadrant 
cohomology spectral sequence:
$$E_2 ^{p,q} = (R^p F) \bigl ( (R^q G )(A) \bigr ) \Rightarrow R^{p+q} (F \circ G) (A).$$
\end{theorem}
 
Applying the Grothendieck Spectral sequence to the composition of the coinduction functor and the hom-functor,
we obtain a spectral sequence which is an ext-group version of  Gabriel-Zisman spectral sequence \cite[Appendix II, Thm 3.6]{GZ-Book}.

\begin{theorem}\label{thm:GZ} Let $F : \cC\to \cD$ be a functor between two small categories.
Let $N$ be an $R\cD$-module and $M$ be an $R\cC$-module. Then there is a first quadrant cohomology spectral sequence
$$E_2 ^{p,q} = \Ext ^p _{R\cD} \bigl ( N, H^q (F/- ; M \circ \pi _{\cC} ) \bigr ) \Rightarrow \Ext ^{p+q} _{R\cC} (\Res_F N, M).$$
\end{theorem} 

\begin{proof}  Let $N$ be a fixed $R\cD$-module. Consider the functors
$$\Coind(-) : R\cC\text{-Mod} \to R\cD\text{-Mod}$$
defined by $M \to \Coind _F M$, and 
$$\Hom _{R\cD} (N, -) : R\cD\text{-Mod} \to R\text{-Mod}$$
defined by $L \to \Hom _{R\cD} (N, L)$. By Proposition  \ref{pro:CoindRightAdjoint}, the composition 
of these functors is isomorphic to the functor 
$\Hom_{R\cC} (  \Res_F N, -)$.
It is clear that both $\Coind_F (-)$ and $\Hom_{R\cD} (N, -)$ are left exact functors.
By Proposition \ref{pro:CoindRightAdjoint}, $\Coind_F (-)$ takes injective $R\cC$-modules to injective $R\cD$-modules. Hence 
we can apply Theorem \ref{thm:Grothendieck} to obtain a spectral sequence 
$$E_2 ^{p,q} = \Ext_{R\cD} ^p \bigl (N,  R^q (\Coind _F (-))  (M) \bigr ) \Rightarrow  \Ext_{R\cC}  ^{p+q} ( \Res_F N, M).$$
By Proposition \ref{pro:RightDerFunc}, the $n$-th right derived functors 
of $\Coind_F(-)$ are isomorphic to the cohomology groups $H^n (F/ - ; M \circ \pi_{\cC} )$. This completes the proof.
\end{proof}

Taking $M=\underline R$, we obtain the following spectral sequence for the cohomology of categories.

\begin{corollary}\label{cor:GZ}
For every $R\cC$-module $M$, there is a spectral sequence
$$E_2 ^{p, q} =H^p (\cD ; H^q (F/-; M \circ \pi _{\cC} ) ) \Rightarrow H^{p+q} ( \cC; M).$$
\end{corollary}

This spectral sequence is usually called the cohomology version of the Gabriel-Zisman spectral sequence. The homology
version is stated in \cite[Appendix II, Thm 3.6]{GZ-Book}. See also \cite{GNT} for the Baues-Wirsching cohomology
version of the Gabriel-Zisman spectral sequence.

 
\subsection{Derived functors of the induction functor}\label{sect:DerivedInduction}

The homology groups of a small category $\cC$ can be interpreted as the derived functors of 
the colimit functor over $\cC$. To explain this, we first recall the definition of the colimit over a category.

\begin{definition} Let $M$ an  $R\cC$-module. The colimit  of $M$ is the $R$-module 
$$\colim _{\cC} M =\Bigl (\bigoplus \limits _{x \in \cC} M(x) \Bigr )/J $$ where 
$J=\langle M(\varphi) (m_x) - m_x \, | \, m_x \in M(x),\ \varphi \in \cC(y, x)\rangle.$
\end{definition} 

It is easy to see that for every $R\cC$-module $M$, we have $\colim _{\cC} M \cong M \otimes _{R\cC}  \underline{R}$.
This shows that the colimit functor  $\colim_{\cC}: R\cC\text{-Mod} \to R\cD\text{-Mod}$  is right exact. The left derived functors of 
the colimit functor are called the \emph{higher colimits} over the category $\cC$, denoted by $\colim ^{\cC} _n (-)$ for $n \geq 0$. 
The following is an easy consequence of  the definitions above.

\begin{lemma}\label{lem:HigherColimits} 
For a left $R\cC$-module $M$, and for every $n \geq 0$, we have $H_n (\cC; M) \cong \colim ^{\cC} _n (M)$.
\end{lemma}

Let $F: \cC \to \cD$ be a functor and $N$ be an $R\cC$-module. By Lemma \ref{lem:IndKanFormula}, 
for every $d \in \Ob (\cD)$, 
\begin{equation}\label{eqn:Isom}
(\Ind _F N )(d) \cong \colim _{ (c, f) \in (d\backslash F)^{op} } (N \circ \pi_{\cC}).
\end{equation}

Since taking colimits is a right exact functor, this isomorphism gives that the induction functor 
$\Ind_F (-)$ is right exact, and the left derived functors of $\Ind _F (-)$
are the higher colimits over the category $(d \backslash F)^{op}$. We conclude the following:

\begin{lemma}\label{lem:AtdIsom}
Let $F: \cC \to \cD$ be a functor. For every $R\cC$-module $N$, and every $d \in \Ob (\cD)$, 
there is an isomorphism
$$L_n (\Ind _F) (N) (d)\cong  H_n  ( d\backslash F ; N \circ \pi_{\cC})$$
where $N \circ \pi_{\cC}$ is the right $R(d\backslash F)$-module defined by the composition 
$N \circ \pi_{\cC} : (d\backslash F) ^{op} \to \cC^{op} \to R$-Mod.
\end{lemma}

\begin{proof} This follows from the isomorphism in (\ref{eqn:Isom}) and from Lemma \ref{lem:HigherColimits}.
\end{proof}

Every morphism $\varphi : d\to d'$ in $\cD$ defines a functor $\varphi\backslash F: d' \backslash F \to d \backslash F$.
By Proposition \ref{pro:FunctorialHom},  this functor induces a graded $R$-module homomorphism  
$$ (\varphi \backslash F)_* : H_* (d' \backslash F ; N \circ \pi_{\cC} ^{d'}) \to H_* (d \backslash F; N \circ \pi _{\cC} ^d). $$
Together with these induced homomorphisms, the assignment $d \to H^n (d \backslash F ; N \circ \pi _{\cC} )$
defines a right $R\cD$-module. We denote this $R\cD$-module with $H_n (-\backslash F ; N \circ \pi_{\cC} )$.
Since the isomorphism in \ref{eqn:Isom} is natural in $d$, the isomorphism in Lemma \ref{lem:AtdIsom}
is natural in $d$. Hence we can conclude the following:

\begin{proposition}\label{pro:LeftDerFunc} Let $F : \cC \to \cD$ be a functor and $N$ be an $R\cC$-module.  
Then for each $n \geq 0$, there is an isomorphism of $R\cD$-modules
$$L_n (\Ind _F (-)) (N) \cong H_n ( -\backslash  F ;  N \circ \pi _{\cC})$$
where $H_n (-\backslash F ; N \circ \pi_{\cC} )$ denotes $R\cD$-module defined above.
\end{proposition}


\subsection{A version of the Gabriel-Zisman spectral sequence}

As a consequence of the results in this section, we obtain the following spectral sequence:

\begin{theorem}\label{thm:VarGZ} Let $F: \cC \to \cD$ be a functor, and let $N$ be a right 
$R\cC$-module and $M$ be a right $R\cD$-module. Then there is a spectral sequence
$$E_2 ^{p,q} = \Ext ^p _{R\cD} ( H_q ( -\backslash F ; N \circ \pi_{\cC} ), M )  \Rightarrow \Ext ^{p+q} _{R \cC } (N, \Res_F M)$$
where $\pi_{\cC} : d\backslash F \to \cC$ is the functor that sends $(c, f) \in \Ob(d\backslash F)$ to $c \in \Ob(\cC)$. 
\end{theorem}

\begin{proof} 
By Proposition \ref{pro:IndLeftAdjoint},  the functor $\Hom _{R\cC} (-; \Res _F M)$ 
is isomorphic to the composition of 
functors $$(\Ind_F(-))^{op}: (R\cC\text{-Mod})^{op} \to (R\cD\text{-Mod})^{op}$$ and 
$$\Hom _{R\cD}(-; M): (R\cD\text{-Mod})^{op} \to R\text{-Mod}.$$
Both of these functors are left exact and the functor $(\Ind _F(-))^{op}$ takes injectives to injectives, because
$\Ind _F (-)$ takes projectives to projectives. Applying Theorem \ref{thm:Grothendieck}, we
obtain a spectral sequence 
$$E_2 ^{p,q} =\Ext _{R \cD} ^p ( R^q ((\Ind _F (-) )^{op} ) (N) , M)  \Rightarrow  \Ext ^{p+q} _{R \cC } (N, \Res _F M).$$
By Proposition \ref{pro:LeftDerFunc},
for every $R\cC$-module $N$,
$$R^q ((\Ind _F (-) )^{op} ) (N) \cong L^q (\Ind _F (-) ) (N) \cong H_q ( -\backslash F  ; N \circ \pi _{\cC} ),$$
hence the proof is complete.
\end{proof}


\section{Isomorphism theorems for the cohomology of categories}

In this section we review some of the isomorphism theorems for the cohomology 
of small categories that we will be using in the rest of the paper.

\subsection{The Jackowski-S\l omi\' nska  theorem} 

One of the important isomorphism theorems for cohomology of small categories is the 
Jackowski-S\l omi\' nska cofinality theorem. We give a proof for this theorem as
a consequence of the spectral sequence in Theorem \ref{thm:VarGZ}.

\begin{definition} Let $R$ be a commutative ring with unity. A small category $\cC$ is called \emph{$R$-acyclic} if 
$H^* (B\cC ; R) \cong H^* (pt; R)$.
A functor $F: \cC \to \cD$ between two small categories is \emph{right cofinal over $R$}
if for every $d \in \Ob (\cD)$, the comma category $d\backslash F$ is 
$R$-acyclic.  
\end{definition}

\begin{theorem}[Jackowski-S\l omi\' nska \cite{JackowskiSlominska}]\label{thm:JackSlom} 
Let $F: \cC \to \cD$ be a functor between two small categories.
If $F$ is right cofinal over $R$, then for every $R\cD$-module $M$, the homomorphism
$$ F^* : H^* (\cD ; M)\xrightarrow{\cong}  H^* (\cC; \Res_F M)$$
induced by the functor $F$, is an isomorphism.
\end{theorem}

\begin{proof} If $F: \cC \to \cD$ is right cofinal over $R$, then for $q\geq 0$,
$$H_q ( -\backslash F ; \underline R  \circ \pi_{\cC } ) \cong H_q (B (-\backslash F ) ; R ) \cong
 \begin{cases} \underline{R} & \text{if  } \ q=0 \\ 0 &\text{if  } \ q>0.\end{cases}$$   
 If we take $N=\underline{R}$ in the spectral sequence in Theorem \ref{thm:VarGZ}, then for every  
$R\cD$-module $M$, this spectral sequence collapses at the $E_2$-page 
to the horizontal line at $q=0$. This gives that
the edge homomorphism $E_2 ^{*, 0} \to H^* (\cC; \Res _F M)$ is an isomorphism.
We have $E_2 ^{*, 0} \cong \Ext ^* _{R\cD} (\underline R , M)\cong H^* (\cD ; M)$. 
It is is easy to show that the edge homomorphism for this spectral sequence
coincides with the map $F^*$ under this identification.
\end{proof}


\subsection{Isomorphisms induced by adjoint pairs}\label{sect:AdjointPair} 

\begin{definition}[{\cite[Def 16.4.1]{Schubert-Book}}]
Let  $F: \cC\to \cD$ and $G: \cD \to \cC$ be functors. The pair $(F, G)$ is called
\emph{a pair of adjoint functors} if there is an isomorphism $$\cD( F(?), ?? ) \maprt{\cong} \cC (?, G(??))$$
of bifunctors $\cC ^{op} \times \cD \to \text{Ens}$ where $\text{Ens}$ is the category of sets in a fixed
universe.
\end{definition}

In the above definition, the $R\cD$-$R\cC$-bimodule
 $R\cD(F(?), ??)$ is the functor $\cC ^{op} \times \cD \to R\text{-Mod}$
 that sends $(c , d) \in \Ob (\cC) \times \Ob (\cD)$ to the $R$-module $R\cD (F(c) ,d)$, and
 the $R\cD$-$R\cC$-bimodule
 $R\cC(?, G(??))$ is the functor $\cC^{op} \times \cD \to R\text{-Mod}$
 that sends $(c , d) \in \Ob (\cC) \times \Ob (\cD)$ to the $R$-module $R\cC (c, G(d))$.
The following is well-known (see \cite[Exercise 2.10]{KashiwaraSchapira-Book}). 
 
\begin{proposition}\label{pro:IndG=ResF} 
Let $(F, G)$ be a pair of adjoint functors where $F: \cC \to \cD$  
and $G: \cD \to \cC$. Then the following holds:
\begin{enumerate}
\item $\Ind _G (-) \cong \Res _F (-)$. As a consequence, $\Ind _G (-)$ is an exact functor. 
\item $\Coind _F(-) \cong \Res _G (-)$. As a consequence, $\Coind _F (-)$ is an exact functor. 
\end{enumerate}
\end{proposition}

\begin{proof} Let $(F, G)$ be an adjoint pair and $N$ be an $R\cD$-module $N$. Then
by Lemma \ref{lem:TensorIsom} we have 
$$\Ind_G (N)=N \otimes _{R\cD} R\cC (?, G(??)) \cong N \otimes _{R\cD} R\cD (F(?), ??)
\cong  \Res _F (N).$$
 
For the second statement, observe that for an $R\cC$-module $M$,  we have 
\begin{align*}
\Coind_F (M)  & \cong \Hom _{R\cC} ( R\cD (F(?), ??), M) \cong  \Hom _{R\cC} ( R\cC (?, G(??)), M)  \\
& \cong M(G(??)) \cong \Res _G (M). 
\end{align*}
\end{proof}

Combining Proposition \ref{pro:IndG=ResF} with Shapiro's lemma in Corollary \ref{cor:CoindShapiro},
we obtain the following change of categories theorem.

\begin{theorem}[{\cite[Lem 5.1]{JMO}}]\label{thm:IsomAdjunction} Let $(F,G)$ be an adjoint pair 
where $F:\cC \to \cD$ and $G:\cD \to \cC$. Then 
for every $R\cC$-module $M$, the homomorphism induced by $G$,
$$G^* : H^* (\cC ; M) \xrightarrow{\cong} H^* ( \cD; \Res_G M),$$
is an isomorphism.
\end{theorem}

\begin{proof} By Proposition \ref{pro:IndG=ResF}, the coinduction functor $\Coind _G(-)$ is exact, hence 
by Lemma \ref{cor:CoindShapiro} there is an isomorphism
$$H^* ( \cD; \Coind _F M) \cong H^* ( \cC ; M).$$
By Proposition \ref{pro:IndG=ResF}, we can replace $\Coind _F M$ with $\Res_G M$. 
By considering a double complex it is easy to show that the induced map is the same 
as the homomorphism induced by $G$.
\end{proof}

Using the isomorphism in Theorem \ref{thm:IsomAdjunction}, in \cite{Jackowski-Transfer}, a transfer map is defined for cohomology
of categories with coefficients in pre-Green functors.

In the case where $F: \cC \to \cD$ is an embedding onto a full subcategory then the following is also true.

\begin{lemma} If $F: \cC \to \cD$ is an embedding onto a full subcategory, then 
\begin{enumerate}
\item $\Res _F \circ \Coind _F=\id_{R\cC}$. 
\item $\Res_F \circ \Ind _F = \id _{R\cC}$.
\end{enumerate}
\end{lemma}

\begin{proof} See \cite[Thm 2.3.3]{KashiwaraSchapira-Book}.
\end{proof}


 \subsection{Equivalence of Categories} 
 
 Recall that two categories $\cC$ and $\cD$ are called \emph{equivalent} if there are  functors $F: \cC\to \cD$ and 
$G : \cD \to \cC$ such that $F\circ G \cong _{\eta} \id _{\cD}$ and $G \circ F \cong _{\mu} \id _{\cC}$ 
where $\eta$ and $\mu$ are natural isomorphisms.  
 
\begin{theorem}[{\cite[Thm 1.11]{BauesWirsching}}]\label{thm:EquivCat}
Let $F: \cC \to \cD$ be an equivalence of two small categories. Then for every 
$R\cD$-module $M$, the map induced by $F$ is an isomorphism
$$F^* : H^* (\cD ; M) \xrightarrow{\cong}
H^* (\cC ;  \Res_F M).$$
\end{theorem}

\begin{proof} In this case the functor $F: \cC\to \cD$ possesses a left adjoint  $L$ (see 
\cite[Thm 2.5.1]{Richter-Book}). Applying Theorem \ref{thm:IsomAdjunction} to the pair $(L, F)$, we
conclude that $F^*$ is an isomorphism.
\end{proof}
 
The category $\cC$ is \emph{skeletal} if any two objects in $\cC$ that are isomorphic are equal. 
A subcategory $\cC'$ of $\cC$ is called a \emph{skeleton} of $\cC$ if it is skeletal and 
if the inclusion functor $i: \cC' \to \cC$ is an equivalence. Every category $\cC$ has a 
skeleton, and up to isomorphism, the skeleton of a category is unique 
(see \cite[Propositions 16.3.4 and 16.3.7]{Schubert-Book}). 
We denote one of these choices of skeletons by $sk(\cC)$ 
and call it the skeleton of the category.  

\begin{corollary}\label{cor:Skeletal} Let $\cC$ be a small category, and let $i: sk(\cC) \to \cC$ 
denote the inclusion of the skeleton of $\cC$. Then for every $R\cC$-module $M$,
we have $H^*(\cC; M)\cong H^* (sk (\cC); \Res_i M) $.
\end{corollary}

This corollary shows that when we are studying the cohomology of small categories,
we can replace the category with its skeleton and assume that it is skeletal. We will do this 
throughout the paper sometimes without mentioning that we are replacing the category with its skeleton.

\subsection{First homotopy property}

In this section we prove a version of the proposition proved in \cite[\S 18]{Penner-Book}, 
called the first homotopy property. We start with some definitions. 

Let $F, G: \cC\to \cD$ be two functors between small categories
and $\eta: F \Rightarrow G$ be a natural transformation from $F$ to $G$. Let $M$ be an $R\cD$-module.
For every morphism $\varphi : x \to y$ in $\cC$, there is a commuting diagram:
$$\xymatrixcolsep{4pc}\xymatrix{M (G(y))  \ar[r]^{M \circ \eta _y}  \ar[d]^{M (G (\varphi)) }  & M(F(y)) \ar[d]^{M(F(\varphi )) } \\ 
  M (G(x))  \ar[r]^{M \circ \eta _x}  & M(F(x)). }$$
This diagram gives that the composition $M \circ \eta$ induces an $R\cC$-module homomorphism
$$\Res_\eta : \Res _G M \to \Res _F M.$$ 
There is a similar homomorphism defined on the induction functors described as follows: The natural transformation $\eta$, defines
an $R\cC$-$R\cD$-bimodule homomorphism
$$R\cD (??, F(?)) \to R\cD (??, G(?))$$
defined by composition with $\eta$. For every $R\cC$-module $M$, this gives an $R\cD$-module homomorphism
$$\Ind _{\eta} : \Ind _F M \cong M \otimes _{R\cC} R\cD( ??, F(?)) \to M \otimes _{R\cC} R\cD (??, G(?)) \cong \Ind _G M.$$
These two homomorphisms commute with adjunction homomorphisms in the following sense:

\begin{lemma}\label{lem:CommDiagram} Let $F, G: \cC\to \cD$ be two functors and 
$\eta: F \Rightarrow G$ be a natural transformation. Then 
for every $R\cC$-module $M$ and $R\cD$-module $N$, the following diagram commutes:
$$
\xymatrix{\Hom _{R\cD} (\Ind _G M, N ) \ar[r]^{\Psi _G \cong}   \ar[d]^{\Hom(\Ind _{\eta}, -)}
& \Hom _{R \cC} (M, \Res_G N) \ar[d]^{\Hom(-, \Res_{\eta} )} \\
\Hom _{R\cD} (\Ind _F M, N ) \ar[r]^{\Psi _F \cong}  &  \Hom _{R \cC} (M, \Res_F N)} 
$$
\end{lemma}

Now we are ready to state the main result of this section.

\begin{proposition}\label{pro:FirstHomotopy} Let $F, G: \cC\to \cD$ be two functors and 
$\eta: F \Rightarrow G$ be a natural transformation. Then for every $R\cD$-module $M$, the 
following diagram commutes:
$$\xymatrix@C=1.6em@R=1em{  &  H^*(\cC ; \Res_G M) \ar[dd]^{ (\Res _{\eta})_* }    \\  H^* (\cD; M) \ar[ru]^{G^*} \ar[rd]^{F^*}  & 
\\  & H^* (\cC; \Res _F M)    }$$
\end{proposition}

\begin{proof} Let $\underline R_{\cC}$ and $\underline R_{\cD}$ denote the constant functors for $\cC$ and $\cD$, respectively.
 For each functor $F: \cC\to \cD$, we have an $R\cC$-module homomorphism $j_F: \underline R_{\cC} \to \Res _F \underline R_{\cD}$ and an $R\cD$-module
homomorphism $\mu_F: \Ind _F \underline R _{\cC} \to \underline R_{\cD}$. These homomorphisms are 
 associated to each other under the isomorphism in Proposition \ref{pro:IndLeftAdjoint}. Let 
 $\eta : F \Rightarrow G$ be a natural transformation, then we have $j_F=\Res _{\eta} \circ j_G $ and $\mu _F=\mu_G \circ \Ind _{\eta}$.

Let $P_*$ be an $R\cC$-projective resolution of $\underline R_{\cC}$, and $Q_*$ be an $R\cD$-projective resolution of $\underline R _{\cD}$.
Let $(\mu _F)_*: \Ind _F P_* \to Q_*$ be a chain map covering  the homomorphism
$\mu _F: \Ind _F \underline R _{\cC}  \to \underline R_{\cD}$, and $(\mu _G)_*: \Ind _G P_* \to Q_*$ be a chain map covering  the homomorphism
$\mu _G: \Ind _G \underline R _{\cC}  \to \underline R_{\cD}$. Since $\mu _F=\mu_G \circ \Ind _{\eta}$, there is a chain homotopy between $(\mu_F)_*$ and
$(\mu _G )_* \circ \Ind_{\eta} $.
For every $R\cD$-module $M$, the homomorphism $F^*$ is defined by the composition
\begin{align*}
H^* (\cD; M) &\cong H^* ( \Hom _{R\cD } (Q_* , M) ) \xrightarrow{(\mu_F)^*} H^* ( \Hom _{R\cD} ( \Ind _F P_* , M)) \\
& \xrightarrow{(\Psi _F) _*\cong}   H^* (\Hom _{R\cC} (P_* , \Res _F M))  \cong H^* (\cC; \Res_F M).
\end{align*}
There is a similar description for the induced map $G^*$ using $(\mu _G )^*$ and $(\Psi _G )_*$.
Using the chain homotopy $(\mu_F)_* \simeq (\mu _G )_* \circ \Ind_{\eta} $
and the commuting diagram in Lemma \ref{lem:CommDiagram}, we conclude
\begin{align*}
(\Res _{\eta} )_* \circ G^* & = (\Res _{\eta} )_* \circ (\Psi_G)_* \circ (\mu _G)^*  
= (\Psi _F )_*  \circ  (\Ind _{\eta })^* \circ (\mu _G)^*  \\
& =(\Psi_F)_* \circ (\mu _F  )^* =F^*.
\end{align*}
\end{proof}

There is also a  homology group version of Proposition \ref{pro:FirstHomotopy} stated using the induced homomorphisms
on homology groups introduced in Proposition \ref{pro:FunctorialHom}. The proof is similar to the cohomology version.


\section{LHS-spectral sequences for extensions of categories}\label{sect:Extensions} 
 
In group theory, an extension of a group $G$ with kernel $K$ 
is defined as a sequence of group homomorphism 
 $$1 \to K \maprt{i} \widehat G \maprt{\pi} G \to 1$$
such that  $i$ is injective and $\pi$ is a surjective with kernel 
equal to $i(K) \subseteq \widehat G$. There is an equivalence relation on such extensions and the equivalence classes 
of extensions of $G$ by $K$ forms a group, denoted by $\mathrm{Ext} (G, K)$. When $K=M$ is an abelian group, then there is a 
$\bbZ G$-module structure on $M$ induced by the conjugations in $\widehat G$. Given a group $G$ 
and a $\bbZ G$-module $M$, the equivalence classes of extensions of $G$ by $M$ is isomorphic 
to the cohomology group $H^2 (G, M)$. There is a similar extension theory for regular extensions of small categories.

\begin{definition}[{\cite[Def 1.1]{Hoff-Extensions}, \cite[Def A.5]{OliverVentura}}]\label{def:RegularExt} 
Let $\cC$ and $\cD$ be two small categories such that $\Ob (\cC)=\Ob(\cD)$. 
Let $\pi: \cC \to \cD$ be a functor that is the identity map on objects and surjective on morphism sets. 
For each $x\in \Ob (\cC)$, set $$K(x) :=\ker \{ \pi_{x,x} : \Aut _{\cC} (x) \to \Aut _{\cD} (x)\}.$$ 
The groups $K(x)$ and $K(y)$ act on the hom-set $\cC(x,y)$ via composition. 
\begin{enumerate}
\item 
The functor  $\pi: \cC \to \cD$ is  \emph{target regular}  if for every 
$x, y \in \Ob(\cC)$,  the left $K(y)$-action on $\cC (x,y)$ is free and 
$\pi _{x,y} : \cC (x,y) \to \cD (x,y)$ is the orbit map of this action. 
\item The functor  $\pi: \cC \to \cD$ is \emph{source regular}  if for every 
$x, y \in \Ob(\cC)$, the right $K(x)$-action on $\cC (x,y) $ is free and 
$\pi _{x,y} : \cC (x,y) \to \cD (x,y)$ is the orbit map of this action. 
\end{enumerate}
\end{definition}

Note that a functor $\pi: \cC\to \cD$ is source regular if the opposite functor $\pi ^{op} : \cC^{op} \to \cD ^{op}$ is target regular.
We will only  introduce the definitions for target regular extensions and assume that the analogous 
definitions for source regular extensions are also introduced.

If $\pi : \cC \to \cD$ is a target regular functor, then the set of subgroups $\cK:=\{ K(x)\}$ over $x\in \Ob (\cC)$  
can be considered as a (discrete)  category 
where $\cK (x, y)=K(x)$ when $x=y$ and empty otherwise. There is a functor 
$i: \cK \to \cC$ that sends every $k\in K(x)$ to the corresponding automorphism in $\Aut_{\cC} (x)$.
A \emph{target regular extension of categories} is a sequence of functors 
$$\cE: \cK \maprt{i} \cC \maprt{\pi} \cD $$
where $\pi$ is a target regular functor. We call the discrete category $\cK$ the \emph{kernel of the extension}. 
When $K(x)$ is an abelian group for every $x \in \Ob(\cC)$, then we say $\cE$ is an extension with abelian kernel. 
We identify the elements $k_x \in K(x)$ with their images under $i_x$ and write $k_x$ for the automorphism
$i_x(k_x)$ in $\Aut _{\cC} (x)$.

\begin{definition} Given a target regular extension $\cE$ as above, there is an induced action of 
$\cC$ on $\cK$ defined as follows:
For every morphism $\varphi: x \to y$ in $\cC$, and for every $k_x  \in K(x)$, 
there is a unique $k_y \in K(y)$ such that  $k_y \circ \varphi =\varphi \circ k_x $.
The assignment $k_x \to k_y$ defines a group homomorphism $K(\varphi): K(x) \to K(y)$. 
We call this \emph{the action of $\cC$ on $\cK$} induced by the extension $\cE$. 
The $\cC$-action on $\cK$ defines a functor $\cC \to Groups$ that sends an object 
$x\in \Ob (\cC)$ to $K(x)$ and a morphism $\varphi : x \to y$ in $\cC$ to  the group 
homomorphism $K(\varphi): K(x)\to K(y)$.  
\end{definition}

For a target regular extension, the action of $\cC$ on $\cK$ induces an action of $\cD$ on $\cK$ modulo the inner automorphisms
of the groups $K(x)$, defined as follows: For every morphism $\varphi : x \to y$ in $\cD$, let $\widetilde \varphi $
be a lifting of $\varphi$ in $\cC$, and let $K(\widetilde \varphi) : K(x)\to K(y)$ be the induced action
of $\widetilde \varphi$ on $\cK$. For a different choice of lifting $\widetilde \varphi$, say $\widetilde \varphi'$, 
we have $\widetilde \varphi '=k_y \circ \widetilde \varphi$ for some $k_y \in K(y)$. This gives that 
$K(\widetilde \varphi ')=c_{k_y} \circ K(\widetilde \varphi)$, hence  
$\varphi\in\cD (x, y)$ induces a well-defined group homomorphism $$K(x)/\Inn (K(x)) \to K(y)/\Inn (K(y)).$$
This defines the $\cD$-action on $\{ K(x)\}$ modulo the inner automorphisms.
If $\cE$ is an extension with an abelian kernel $\cK=\{M (x)\}$, then every morphism $\varphi : x \to y$ in $\cD$ induces a well-defined 
group homomorphism $M(\varphi) : M(x)\to M(y)$ which defines a left $\bbZ \cD$-module
structure on $M$. For a source regular extension, in a similar way we obtain a right $R\cD$-module structure on $M$.

\begin{definition}[{\cite[p. 20]{Webb-Survey}}] Two target regular extensions $\cE_1$ and $\cE_2$ are \emph{equivalent} 
if there is functor $F: \cC_1 \to \cC_2$ such that the following diagram commutes. 
$$\xymatrix@C=2.1em@R=1.2em{ & \cC_1 \ar[dd]^F \ar[dr]^{\pi_1} & \\   
\cK \ar[ur]^{i_1} \ar[dr]_{i_2}&  & \cD  \\ & \cC_2 \ar[ur]_{\pi_2} & }$$
In this case $F$ is an equivalence of categories and this relation is an
equivalence relation. 
\end{definition}

We have the following classification theorem for equivalence classes of regular extensions.

\begin{theorem} Let $\cD$ be a small category and $M$ be a left $\bbZ \cD$-module. 
The equivalence classes of target regular extensions $$\cE: M \maprt{i} \cC \maprt{\pi} \cD $$ 
with kernel $M$ is in one-to-one correspondence with the cohomology classes in $H^2 (\cD, M)$.
\end{theorem}

In the rest of the section we prove the necessary preliminary results to introduce the LHS-spectral sequence for target and 
source regular extensions. Let $\cE : \cK \maprt{i} \cC \maprt{\pi} \cD $ be a target regular extension. 
Then for every $x\in \Ob (\cD)$,  
there is a functor $$J_x : K(x) \to \pi/ x $$ from the group category $K(x)$ to the comma category $\pi/x$.
The functor $J_x$ sends the unique object $x$ of $K(x)$ to $(x, \id_x)$ in $\Ob (\pi/x)$, and sends a group element 
$k\in K(x)$ to the morphism $k: (x, \id_x) \to (x, \id_x)$ in $\pi/x$ defined by $k : x \to x$ in $i(K(x))\leq \Aut _{\cC} (x)$. 
The functor $J_x$ is defined in \cite[p. 9]{Quillen-KTheory} as part of the discussion on prefibred categories.
 
Let $M$ be a right $R\cC$-module. For every $x\in \Ob (\cD)$, consider the $RK(x)$-module $$\Res _{J_x}\Res _{ \pi _{\cC}} M$$
where $\pi _{\cC}: \pi/x \to \cC$ is the functor that sends $(c, \varphi: c \to x)$ in $\pi/x$ to $c\in \Ob (\cC)$.
Since the composition $\pi_{\cC} \circ J_x$ is equal to the inclusion functor $i_x : K(x) \to \cC$, 
there is an isomorphism of $RK(x)$-modules $\Res _{J_x} \Res _{ \pi _{\cC}} M \cong \Res _{i_x} M$. 
We have the following:

\begin{lemma}\label{lem:IsomComma} Let $\cE: \cK \to \cC \maprt{\pi} \cD $ 
be a target regular extension and $M$ be an $R\cC$-module.
Then for every $x\in \Ob (\cC)$,  the homomorphism
$$J_x ^* : H^* (\pi / x; \Res _{\pi _{\cC}} M )  \xrightarrow{\cong} H^* (K(x); \Res _{i_x} M )$$  
induced by the functor $J_x : K(x) \to \pi/x$ is an isomorphism.
\end{lemma}

\begin{proof}  In \cite{Xu-CohSmall}, this is proved by showing that the functor $J_x$ is right cofinal 
(see \cite[Lem 4.2]{Xu-CohSmall}). Then the result follows from the Jackowski-S\l omi\' nska cofinality theorem. 
Alternatively, one can show that the functor $J_x$ has a left adjoint, and then  one can apply 
Theorem \ref{thm:IsomAdjunction} to conclude that the induced map $J_x^*$ is an isomorphism. 

The left adjoint $r_x: \pi /x \to K(x)$ of $J_x$ is defined as follows: Since $K(x)$ has a unique object 
$x\in K(x)$, $r_x$ sends every object of $\pi /x$ to $x$.
For each morphism $y \to x$ in $\cD$, choose a lifting $\widetilde \varphi :y \to x$ in $\cC$. 
Let $\alpha: (y_1, \varphi_1: y_1 \to x) \to (y_2 , \varphi_2 : y_2 \to x)$ be a morphism in $\pi/x$ 
defined by a morphism $\alpha : y_1 \to y_2$  in $\cC$ satisfying $\varphi _1=\varphi_2 \circ \pi (\alpha)$. 
Then there is a unique $k_{\alpha} \in K(x)$ such that $k_{\alpha} \circ \widetilde \varphi_1= \widetilde \varphi _2 \circ  \alpha$. We define $r_x(\alpha):=k_{\alpha}$.
For every $\beta : (y_2, \varphi_2: y_2 \to x)  \to (y_3, \varphi_3 : y_3 \to x)$, we have 
$$\widetilde \varphi_3 \circ (\beta \circ \alpha) = ( \widetilde \varphi_3 \circ \beta) \circ \alpha 
=(k_{\beta} \circ \widetilde \varphi _2) \circ \alpha =
k_{\beta} \circ (\widetilde \varphi _2 \circ \alpha)=k_\beta k_\alpha \circ \widetilde \varphi_1.$$
By the uniqueness of $k_{\beta \circ \alpha} \in K(x)$, we have $k_{\beta \circ \alpha}=k_\beta k_\alpha$. 
Hence $r_x$ defines a functor.

It is clear that $r_x \circ J_x=\id _{K(x)}$. We need to show that there is a natural transformation  
$\eta: \id _{\pi/x} \Rightarrow J_x \circ r_x$. We define
$$\eta _{(y, \varphi : y \to x)} : (y, \varphi: y \to x) \to (x, \id _x : x\to x )$$
to be the morphism defined by the morphism $\widetilde \varphi : y \to x$ in $\cC$. 
It is easy to check that this defines a natural transformation.
\end{proof} 

\begin{remark} Note that for any $R\cC$-module $M$, the $K(x)$-module $\Res _{i_x} M$ can be thought
as the restriction of the $R\Aut _{\cC} (x)$-module $M(x)$ to $K(x)$ via the inclusion map
$K(x) \to \Aut _{\cC} (x)$. When we think of $\Res _{i_x} M$ in this way we write it as $M(x)|_{K(x)}$. 
We sometimes just write $M(x)$ for this $K(x)$-module when it is clear from the context that
we mean the $K(x)$-module. 
\end{remark}

We will now show that the isomorphism in Lemma \ref{lem:IsomComma} is actually an isomorphism 
of right $R\cD$-modules with suitably defined $R\cD$-module structures on both sides of the isomorphism.
On the left side of the isomorphism we will use the $R\cD$-module $H^* (\pi /- ; \Res _{\pi_{\cC}} M)$ introduced
in Lemma \ref{lem:RD-ModuleStructure}.

For the right hand side we define the $R\cC$-module structure as follows: 
Let $\varphi : x\to z$ be a morphism in $\cD$, and let $\widetilde \varphi : x \to y$ be a lifting of $\varphi$ in $\cC$.   
Let $K(\widetilde \varphi): K(x) \to K(y)$ denote the group homomorphism induced by the $\cC$-action on $\cK$.
Note that the $R$-module homomorphism $M (\widetilde \varphi)$ gives an $RK(x)$-homomorphism
$\Res _{K(\widetilde \varphi)} M(y) \to M(x)$.

\begin{lemma}\label{lem:InducedHom}  Let $\cE : \cK \maprt{i} \cC \maprt{\pi} \cD $ 
be a target regular extension and $M$ be an $R\cC$-module.
Let $\varphi : x\to y$ be a morphism in $\cD$. 
Consider the following composition of graded $R$-module homomorphisms
$$H^* (K(y); M(y) ) \xrightarrow{K(\widetilde \varphi)^* } H^* (K(x) ; \Res _{K(\widetilde \varphi) } 
M(y) ) \xrightarrow{M(\widetilde \varphi)_* }
H^* (K(x); M(x)).$$
where the first map is the homomorphism induced by the group
homomorphism $K(\widetilde \varphi)$ and the second map is the homomorphism 
induced by the homomorphism $M(\widetilde \varphi): M(y)\to M(x)$.
Then the composition $M(\widetilde \varphi)_* K(\widetilde \varphi)^*$ 
does not depend on the choice of the lifting $\widetilde \varphi $ for $\varphi$.
\end{lemma}

\begin{proof} If $\widetilde \varphi '$ is another lifting of $\varphi :x \to y$, then 
$\widetilde \varphi'=k_y \circ \widetilde \varphi$ for some $k_y \in K(y)$.
From the definition of $K(\widetilde \varphi)$, we see that 
$K(\widetilde \varphi ')= c_{k_y} \circ K(\widetilde \varphi)$ where $c_{k_y}: K(y) \to K(y)$
is the conjugation map defined by $c_{k_y} (u)=k_y u k_y ^{-1}$. 
The composition $K(\widetilde \varphi)^* M(\widetilde \varphi) _*$ is the map induced by 
a pair $(K(\widetilde \varphi ), M(\widetilde \varphi))$ in the sense described 
in \cite[Sect III.8]{Brown-Book}. We have 
\begin{equation}
\begin{split}
M(\widetilde \varphi ')_* K(\widetilde \varphi ')^* &= (M(\widetilde \varphi) \circ 
M(k _y ) ) _*  (c_{k_y} \circ K(\widetilde \varphi ) )^*
=M(\widetilde \varphi)_*  M(k_y)_*  K(\widetilde \varphi )^*  c_{k_y}^*\\
&=M(\widetilde \varphi)_* K(\widetilde \varphi )^* M(k_y) _* c_{k_y}^*.
\end{split}
\end{equation}
The composition $M(k)_* c_k^*$ defines the action of $k \in K(x)$ on  $H^*(K(x); M(x))$. 
By \cite[Prop II.8.3]{Brown-Book}, this action is trivial, hence the composition 
$M(\widetilde \varphi)_* K(\widetilde \varphi )^*$ does not depend on the choice of the lifting $\widetilde \varphi$.
\end{proof}

\begin{definition}
The assignment $d \to H^* (K(d) ; M (d))$ together with the induced 
homomorphisms $H^* (K(y); M(y))\to H^* (K(x); M(x))$ introduced in Lemma \ref{lem:InducedHom}  defines 
a right $R\cD$-module. We denote this right $R\cD$-module by $H^* ( \cK ; M)$. 
\end{definition}
 
Now we are ready to prove:

\begin{lemma}\label{lem:IsomCommaFunc} With the definitions given above, the isomorphism 
in Lemma \ref{lem:IsomComma} induces an isomorphism of right 
$R\cD$-modules $$H^* (\pi /- ;  \Res_{\pi_{\cC} } M ) \cong H^* (\cK ; M).$$
\end{lemma}

\begin{proof} We need to show that for every morphism $\varphi : x \to y$ in $\cD$, 
the following diagram commutes
$$\xymatrixcolsep{4pc}\xymatrix{ H^* (\pi/y; \Res _{\pi_{\cC} ^y} M ) \ar[r]^{(\pi /\varphi ) ^*} 
\ar[d]^{J_y ^*} &  H^* (\pi /x ; \Res_{\pi _{\cC} ^x}  M) \ar[d]^{J_x ^* } \\
H^* (K(y); M (y) ) \ar[r]^{M(\widetilde \varphi)_* K(\widetilde \varphi )^*} &  H^* ( K(x)  ;  M(x)).} 
$$
Let $\varphi : x\to y$ be a morphism in $\cD$. Consider the following diagram: 
$$\xymatrix{ K(x) \ar[r]^{J_x} \ar[d]^{K(\widetilde \varphi)} & \pi / x  \ar[d]^{\pi / \varphi }   
\\ K(y)  \ar[r]^{J_y} & \pi /y }$$
This diagram does not commute but there is a natural transformation $\eta : F \Rightarrow G$ where $F= \pi /\varphi \circ J_x$
and $G=J_y \circ K(\widetilde \varphi )$. The functors $F$ and $G$ send the unique object $x$ of $K(x)$ to 
$F(x)=(\pi/\varphi)((x, \id_x))=(x, \varphi: x\to y)$
and $G(x) =J_y (y)=(y, \id_y: y \to y)$. We define $\eta _x: F(x) \to G(x)$ to be the morphism given 
by $\widetilde \varphi : x \to y$. For $k_x \in K(x)$, $F(k_x)=(k_x: (x, \varphi)\to (x, \varphi))$ and 
$G(k_x)=(k_y: (y, \id_y)\to (y, \id_y))$ where $k_y =K(\widetilde \varphi )(k_x)$.
Since $k_y \circ \widetilde \varphi =\widetilde \varphi \circ k_x$, we have 
$G(k_x) \circ \eta _x =\eta_x \circ F(k_x)$. So, $\eta$ is a natural
transformation from $F$ to $G$.

Applying Proposition  \ref{pro:FirstHomotopy} to $\eta: F \Rightarrow G$, we obtain that $(\Res_{\eta} )_* \circ G^*=F^*$. This gives that
$$(\Res _{\eta} )_* \circ K(\widetilde \varphi )^* \circ J_y ^* =J_x^* \circ  (\pi /\varphi)^*.$$
The commutativity of the above cohomology diagram follows from the fact  that
the homomorphism $\Res _{\eta} : \Res _G \Res _{\pi_{\cC}} M \to \Res _F \Res _{\pi _{\cC} } M $ coincides with the 
$RK(x)$-module homomorphism $M(\widetilde \varphi) : \Res _{K(\widetilde \varphi ) } M(y) \to M(x)$.    
\end{proof}

Combining Lemma \ref{lem:IsomCommaFunc} with the Gabriel-Zisman spectral sequence proved in 
Theorem \ref{thm:GZ} gives the following theorem which is the ext-group version of LHS-spectral sequence 
due to Xu \cite{Xu-CohSmall}.

\begin{theorem}\label{thm:XuSSExt} Let $\cE : \cK \maprt{i} \cC \maprt{\pi} \cD $ be a target regular extension.
Then for every $R\cC$-module $M$ and $R\cD$-module $N$, there is a spectral sequence
$$ E_2 ^{p,q}= \Ext ^p _{R\cD} (N,  H^q ( \cK ; M) ) \Rightarrow \Ext ^{p+q} _{R \cC} ( \Res_{\pi} N, M).$$
\end{theorem}

\begin{proof} 
By Theorem \ref{thm:GZ}, for every $R\cC$-module $M$ and $R\cD$-module $N$, there is a spectral sequence
$$E_2 ^{p, q} =\Ext^p _{R \cD} (N, H^q ( \pi /-; M \circ \pi _{\cC} ) ) \Rightarrow \Ext^{p+q} _{R\cC} (\Res _{\pi} N,  M).$$
By Lemma \ref{lem:IsomCommaFunc}, we can replace  the $R\cD$-module
$H^* (\pi /-; M \circ \pi _{\cC} )$ with $H^* (\cK ; M).$
This gives the desired spectral sequence.
\end{proof}

If we take $M=\underline{R}$ in Theorem \ref{thm:XuSSExt}, we recover the LHS-spectral sequence due to Xu.

\begin{theorem}[Xu, \cite{Xu-CohSmall}]\label{thm:XuSSCoh} Let $\cE : \cK \to \cC \to \cD $ be a target regular extension.
Then for every $R\cC$-module $M$,  there is a spectral sequence
$$ E_2 ^{p,q}= H^p ( \cD;  H^q ( \cK; M)) \Rightarrow H^{p+q} ( \cC ;  M).$$
\end{theorem}

We call this spectral sequence the LHS-spectral sequence for a target regular extension. This spectral sequence is also 
constructed in \cite[Prop A.11]{OliverVentura}. 

\begin{remark} A small category $\cC$ is called a finite category if the set of all morphisms in $\cC$ is finite.
In \cite[Prop 4.4]{Xu-OnLocal}, Xu shows that if $\cC$ is a finite category and $k$ is a field, then 
for every $k\cC$-module $M$ and $N$, there is an isomorphism 
$$ \Ext ^* _{k\cC} (M, N) \cong \Ext ^* _{k F(\cC)} (\underline k , \Res _{\nabla } \Hom _k (M, N) ),$$
where $F(\cC)$ denotes the category of factorizations in $\cC$ (see \cite[p. 88]{BauesWirsching}) and  
$\nabla : F(\cC) \to \cC \times \cC ^{op}$ is the  functor that takes $\varphi: x \to y$ in $F(\cC)$ to the pair $(y,x)$ in
$\cC \times \cC^{op}$. This suggests that with some extra work it should be possible to derive 
the ext-group version of the LHS-spectral sequence from the cohomology version when $\cC$ is a finite category. 
 \end{remark}

For source regular extensions, there is a similar spectral sequence with coefficients 
in a left $R\cC$-module $M$ obtained by applying Theorem \ref{thm:XuSSExt} 
to the opposite extension $\cE^{op}$.

\begin{corollary}\label{cor:XuSSExt-Source} Let $\cE : \cK \maprt{i} \cC \maprt{\pi} \cD $ 
be a source regular extension. Then for every left $R\cC$-module $M$ and left $R\cD$-module 
$N$, there is a spectral sequence $$ E_2 ^{p,q}= \Ext ^p _{R\cD} (N,  H^q ( \cK ; M) ) 
\Rightarrow \Ext ^{p+q} _{R \cC} ( \Res_{\pi} N, M)$$
where $H^q ( \cK; M)$ is the left $R\cD$-module whose structure is defined to be 
the right $R\cD^{op}$-structure of $H^q (\cK; M)$ coming from the opposite extension 
$\cE^{op}: \cK \to \cC^{op} \to \cD^{op} $.
\end{corollary}

Note that if $\cE : \cK \maprt{i} \cC \maprt{\pi} \cD$ is a source regular extension and if we want to
relate the cohomology of $\cC$ and $\cD$ with right module coefficients, the spectral sequences above 
cannot be used. For this we construct a spectral sequence involving the homology groups of the groups $K(x)$. 
We now explain this spectral sequence:

Let $\cE:  \cK \maprt{i} \cC \maprt{\pi} \cD $ be a source regular extension.  For every 
$\varphi : x\to y$ in $\cD$, choose a lifting $\widetilde \varphi : x \to y$ in $\cC$. 
The morphism $\widetilde \varphi$ induces a group homomorphism 
$K(\widetilde \varphi): K(y) \to K(x)$ defined by $K (\widetilde \varphi )(k_y)=k_x$ 
where $k_x\in K(x)$ is the unique element satisfying $k_y \circ \widetilde \varphi=\widetilde \varphi \circ k_x$. 
For every $q\geq 0$, this induces homomorphisms
$$ H_q (K(y); M(y) ) \xrightarrow{K(\widetilde \varphi)_* } H_q (K(x) ; \Res _{K(\varphi) } M(y) ) \xrightarrow{M(\varphi)_* }
H_q (K(x); M(x)).$$
One can show  in a similar way to the way it was done in Lemma \ref{lem:InducedHom} that this homomorphism 
does not depend on the lifting $\widetilde \varphi$ for $\varphi$. Hence it defines a right $R\cD$-module structure 
for the assignment $d \to H_q (K(d), M(d))$, denoted by $H_q (\cK; M)$. 

\begin{lemma}\label{lem:IsomCommaFuncHomology}  There is an isomorphism of right 
$R\cD$-modules $$H_* (-\backslash \pi ;  \Res_{\pi_{\cC} } M ) \cong H_* (\cK ; M).$$
\end{lemma}

\begin{proof} A proof can be given by repeating the arguments in the proof
of Lemma \ref{lem:IsomCommaFunc} for homology groups and using the 
homology groups version of  Proposition  \ref{pro:FirstHomotopy}.
\end{proof}

We have the following theorem.
  
\begin{theorem}\label{thm:SSGYGen}
Let $\cE :  \cK \maprt{i} \cC \maprt{\pi}\cD $ be a source regular extension. For every right
$R\cD$-module $M$ and every right $R\cC$-module $N$,  
there is a spectral sequence
$$ E_2 ^{p,q}= \Ext ^p _{R \cD} ( H_q ( \cK; N), M) \Rightarrow \Ext ^{p+q} _{R\cC} (N, \Res _{\pi} M).$$
Here $H_q ( \cK; N)$ is the right $R\cD$-module whose structure is introduced above.
\end{theorem}

\begin{proof}   
The proof follows from the spectral sequence in Theorem \ref{thm:VarGZ} and from 
Lemma \ref{lem:IsomCommaFuncHomology}.
\end{proof}

 If we take $M=\underline{R}$,  we obtain the following:
 
 \begin{corollary}[G\" undo\u gan-Yal{\c c}{\i}n \cite{GundoganYalcin}]\label{cor:SSGY} 
Let $\cE :  \cK \maprt{i} \cC \maprt{\pi} \cD $ be a source regular extension. 
Then, for every right $R\cD$-module $M$,
there is a spectral sequence
$$ E_2 ^{p,q}= \Ext^p _{R \cD} ( H_q ( \cK; R) , M ) \Rightarrow H^{p+q} (\cC ; \Res _{\pi} M).$$ 
\end{corollary}
 
We call this spectral sequence
LHS-spectral sequence for a source regular extension.  This spectral sequence was used 
in \cite[Thm 1.3]{GundoganYalcin}  to relate the cohomology of a $p$-local finite group 
$(S, \cF, \cL)$ to the cohomology of the transporter category $\cT^c _S(\pi)$ 
for the infinite group $\pi$ realizing the fusion system $\cF$.

 
 \section{Spectral sequences for regular EI-categories}\label{sect:RegularEI}

\begin{definition} A small category $\cC$ is called an \emph{EI-category} if every endomorphism in $\cC$ is an isomorphism.
\end{definition}

Many interesting categories that appear in representation theory, such as the orbit category, the fusion systems, 
and the linking systems are EI-categories. We introduce these examples of EI-categories 
in Sections \ref{sect:Transporter} and  \ref{sect:Linking}.

Let $\cC$ be an EI-category. For every $x \in \Ob (\cC)$, we denote the isomorphism class of $x$ by $[x]$ and the 
set of isomorphism classes of objects in $\cC$ by $[\cC]$.
The set $[\cC]$ is a partially ordered set with the order relation given by 
$$[x]\leq [y] \ \text{  if   } \  \cC (x, y) \neq \emptyset.$$ 
We consider the partially ordered set $[\cC]$ as a category with morphisms given by the order relation.
There is a functor $\pi: \cC\to [\cC]$ that sends each $x\in \cC$ to its isomorphism class $[x]$, and each morphism
$\varphi : x\to y$ to the order relation $[x]\leq [y]$.

In general the functor $\pi: \cC \to [\cC]$ does not define a target (or source) regular extension of categories. 
The first problem is that the object sets of $\cC$ and $[\cC]$ are not equal in general. We can remedy 
this by replacing $\cC$ with its skeletal category $sk (\cC)$.
By Corollary \ref{cor:Skeletal}, this adjustment does not change the cohomology of the category $\cC$. 
\emph{Throughout the paper, we will assume that every EI-category $\cC$ is skeletal, and we identify the 
equivalence class of each object $x \in \Ob (\cC)$ with itself and take $\Ob ([\cC]) = \Ob (\cC)$. 
 We write $x\leq y$ whenever $\cC(x, y)\neq \emptyset$.}
 
\begin{definition} An EI-category $\cC$ is called \emph{target regular} if for every $x, y \in \Ob (\cC)$, 
the automorphism group $\Aut _{\cC} (y)$ acts regularly (freely with one orbit)
on $\cC (x, y)$. An EI-category $\cC$ is called \emph{source regular} if for every $x, y \in \Ob (\cC)$, 
$\Aut _{\cC} (x)$ acts regularly on $\cC (x, y)$. 
\end{definition}

By definition, every (skeletal) EI-category $\cC$ is target (resp. source) regular if and only if 
the functor $\pi : \cC \to [\cC]$ is target (resp. source) regular. Hence the spectral sequences we obtained 
in the previous section applies to the functor $\pi: \cC \to [\cC]$ and gives the following spectral sequences.
 
\begin{theorem}\label{thm:RegularSS} 
\begin{enumerate}
\item Let $\cC$ be a target regular category and $M$ be an $R\cC$-module. Then, for every $n \geq 0$, the assignment 
$x \to H^n ( \Aut_{\cC} (x); M(x))$ defines an $R[\cC]$-module and 
there is a spectral sequence
$$ E_2 ^{p,q}= H^p ( [\cC];  H^q ( \Aut_{\cC} (-); M(-) ) ) \Rightarrow H^{p+q}  ( \cC ;  M).$$ 
\item Let $\cC$ be a source regular category and $M$ be a left $R\cC$-module. Then, for every $n \geq 0$, the assignment 
$x \to H^n ( \Aut_{\cC} (x); M(x))$ defines a left $R[\cC]$-module and 
there is a spectral sequence
$$ E_2 ^{p,q}= H^p ( [\cC];  H^q ( \Aut_{\cC} (-); M(-) ) ) \Rightarrow H^{p+q}  ( \cC ;  M)$$ 
where all the cohomology groups are cohomology with coefficients in a left module. 
\item Let $\cC$ be a source regular category and $M$ be a right $R[\cC]$-module. Then, for every $n \geq 0$, the assignment 
$x \to H_n ( \Aut_{\cC} (x); R)$ defines a right $R[\cC]$-module and 
there is a spectral sequence
$$ E_2 ^{p,q}= \Ext _{R[\cC] } ^p ( H_q ( \Aut_{\cC} (-); R ), M ) \Rightarrow H^{p+q}  ( \cC ;  \Res _{\pi} M).$$ 
\end{enumerate}
\end{theorem}

\begin{proof} The first two spectral sequences follow from Theorem \ref{thm:XuSSCoh} 
and Corollary \ref{cor:XuSSExt-Source}. The third one follows from Corollary \ref{cor:SSGY}.
\end{proof} 

For an abelian group $A$, the cohomology of the category $\cC$ with coefficients in $A$ is defined
as the cohomology of the classifying space $B\cC$ with coefficients in $A$. The cohomology group 
$H^* (B\cC; A)$ is isomophic to the cohomology of the category $\cC$ with coefficients 
in the constant functor $\underline A$. Here we can regard the constant functor as either a left
module or a right $R\cC$-module, with both cohomology groups being isomorphic to the cohomology 
group $H^* (B\cC; A)$. This gives the following spectral sequence for  source regular categories.
 
\begin{corollary}[{Linckelmann \cite[Thm 3.1]{Linckelmann-OnH}}]\label{cor:RegularSS} 
Let $\cC$ be a source regular category and $A$ be an abelian group. 
Then, for every $n \geq 0$, the assignment 
$x \to H^n ( \Aut_{\cC} (x); A)$ defines a left $R[\cC]$-module and 
there is a spectral sequence
$$ E_2 ^{p,q}= H^p ( [\cC];  H^q ( \Aut_{\cC} (-); A ) ) \Rightarrow H^{p+q}  ( B\cC ;  A).$$ 
\end{corollary}

\begin{proof} This follows from the Statement (2) in Theorem  \ref{thm:RegularSS} after taking the left $R\cC$-module 
$M$ as the constant functor $\underline A$.
\end{proof}
  
One of the main examples of the source regular category is the subdivision category $S(\cC)$ of an EI-category $\cC$. 
We now give the definition of the subdivision category, following the terminology in \cite{Linckelmann-Orbit}.  
 
The simplex category $\Delta$ is the category whose objects are the totally ordered sets $[n]=\{ 0, 1, \dots, n\}$, where $n\geq 0$ 
is an integer, and the morphisms $[m]\to [n]$ are given by order preserving maps for all $m,n \geq 0$.  Let $\Delta _{inj} \subset \Delta$
denote the subcategory that has same objects as $\Delta$ where the morphisms $[m]\to [n]$ are given by 
order preserving injective maps for all $m,n \geq 0$.  

Let $\cC$ be a small category. A functor $\sigma : [n]\to \cC$ is a
chain of composable morphisms in $\cC$:  $$ \sigma := ( \sigma_0 
\xrightarrow{\alpha_1} \sigma_1 \xrightarrow{\alpha_2} \cdots  \to \sigma_{n-1} \xrightarrow{\alpha_n} \sigma_n )$$
where $\sigma (i)$ is denoted by $\sigma_i$ in the chain.
 We say a chain $\sigma$ is \emph{strict} if the objects $\sigma_i$ in the chain are pairwise non-isomorphic, 
 i.e., distinct, since we assumed that $\cC$ is skeletal.

\begin{definition}[{\cite[\S 2]{Linckelmann-Orbit}}]\label{def:Subdivision} Let $\cC$ be a small category. 
The \emph{subdivision category} $S(\cC)$ of $\cC$ 
is the category whose objects are strict chains $\sigma : [n] \to \cC$, where a morphism from 
$\tau : [m]\to \cC$ to $\sigma : [n] \to \cC$ is given by a pair $(j, \mu)$ such that $j: [m]\to [n]$ 
is a morphism in $\Delta _{inj}$ and 
$\mu : \tau \cong  \sigma \circ j$ is a natural isomorphism. The composition of two morphisms
 $(k, \nu) :\rho \to \tau$  and $(j, \mu) : \tau \to \sigma$ in $S(\cC)$ is defined by 
 $$ (j, \mu) \circ  (k, \nu) = (j \circ k, (\mu k) \circ \nu)$$
where $\mu k: \tau \circ k \cong \rho \circ j \circ k$ is the natural isomorphism obtained 
by precomposing $\mu$ with $k$. The opposite category $S(\cC)^{op}$ is denoted by $s(\cC)$.
\end{definition}

If $\cC$ is a poset considered as a category, then the subdivision category $S(\cC)$ is the poset category 
of the strict chains in $\cC$, with the order relation given by the inclusion of subchains.
In this case it is well-known that the realizations $B\cC$ and $BS(\cC)$ are homeomorphic 
by the well-known result on the subdivision of a simplicial complex. There is a similar observation for general 
subdivision categories obtained via a functor from $S(\cC)$ to $\cC$ which we now explain.

Let $\Fin: S(\cC)\to \cC$ denote the functor that sends
$\sigma : [n] \to \cC$ to the last object $\sigma_n$ in the chain and sends  a morphism  
$(j, \mu): \tau \to \sigma$ to the composition of morphisms
$$ \tau _m \xrightarrow{\mu _j} \sigma _{ j(m)}  \xrightarrow{\alpha_{j(m)+1} } 
\sigma_{j(m)+1} \xrightarrow{\alpha_{j(m)+2}} \cdots 
\xrightarrow{\alpha_n} \sigma _n.$$
Similarly one can define the functor $\Ini: S(\cC) ^{op} \to \cC$ that sends
$\sigma : [n] \to \cC$ in $S(\cC)$ to the first object $\sigma_0$ in the chain and sends 
a morphism  $(j, \mu): \tau \to \sigma$ to the composition of the following morphisms
$$\sigma _0  \xrightarrow{\alpha_1} \sigma _1 \xrightarrow{\alpha_2} \cdots \xrightarrow{\alpha_{j(0)} } \sigma _{j(0)}
\xrightarrow{\mu_0 ^{-1} } \tau _0.$$ 

Note that we can view the functor $\Ini$ as a (covariant) functor $\Ini : s(\cC) \to \cC$.
For this functor, we have the following observation due to S\l omi\' nska
\cite{Slominska-Homotopy}.

\begin{proposition}\label{thm:SlomHom} 
For any EI-category $\cC$, the functor $\Ini: s(\cC) \to \cC$ is right cofinal.  
\end{proposition}

\begin{proof} See \cite[Prop 1.5]{Slominska-Homotopy}.
\end{proof}

Applying Theorem \ref{thm:JackSlom}, we obtain: 

\begin{proposition}\label{pro:SC-Cohomology}
Let $\cC$ be an $EI$-category. Then for any $R\cC$-module $M$,
the map $$\Ini^* : H^* (\cC; M ) \xrightarrow{\cong} H^* (s(\cC) ; \Res _{\Ini} M) $$
induced by the functor $\Ini: s(\cC) \to \cC$ is an isomorphism.
\end{proposition}

\begin{proof}  This follows from Proposition \ref{thm:SlomHom} and Theorem \ref{thm:JackSlom}.
\end{proof}

Proposition \ref{pro:SC-Cohomology} shows that for the purposes of calculating cohomology groups 
of the category $\cC$, we can replace $\cC$ with the category $s(\cC)$. Moreover the spectral sequence 
in Theorem \ref{thm:RegularSS} applies the functor $\pi : s(\cC) \to [s(\cC)]$. For this we first need 
to replace $s(\cC)$ with its skeleton $sk(s(\cC))$ and assume that
$\Ob (s(\cC) )=\Ob ([s(\cC)])$.  Under this assumption the following holds:

\begin{proposition}\label{pro:SubIsTargetRegular}\label{pro:SC-SourceRegular}
The subdivision category $S(\cC)$ is a source regular category. Hence $s(\cC)$ is target regular.
\end{proposition}  

\begin{proof} See \cite[Prop 2.3]{Linckelmann-Orbit}.
\end{proof}

Note that since we assumed that $s(\cC)$ is skeletal, there is a unique pre-chosen object for every
isomorphism class of objects. We write isomorphism classes as $[\sigma]$ meaning that $\sigma$ is the chosen object.  
For each $\sigma\in s(\cC)$, let $\Aut _{s(\cC) } (\sigma)$ denote the automorphism group 
of $\sigma$. If  $$ \sigma := ( \sigma_0 
\xrightarrow{\alpha_1} \sigma_1 \xrightarrow{\alpha_2} \sigma_2 \to\cdots \to \sigma_{n-1} \xrightarrow{\alpha_n} \sigma_n )$$
then an automorphism of $\sigma$ is a sequence of automorphisms $(a_i)$ with $a_i \in \Aut _{\cC} (\sigma_i)$ such that
$$ a_i \circ \alpha_i =\alpha_i \circ a_{i-1}$$ for each $i=1, \dots, n-1.$
Let $\pi_i : \Aut _{s(\cC)} (\sigma) \to \Aut _{\cC} (\sigma _i)$ denote the group homomorphism that
sends $(a_i) \in \Aut _{s(\cC) } (\sigma)$ to $a_i \in \Aut _{\cC} (\sigma_i) $. 

Let $M$ be an $R\cC$-module and $\sigma \in s(\cC)$. Note that $M(\sigma_0)$ is 
an $\Aut _{\cC} (\sigma_0)$-module. Restriction of this module via 
$\pi_0: \Aut _{s(\cC)} (\sigma) \to \Aut _{\cC} (\sigma _0)$ gives an $R\Aut _{s(\cC)} (\sigma)$-module
$\Res _{\pi_0 } ( M(\sigma_0))$. We have the following:

\begin{lemma}\label{lem:Modules} With the definitions given above, for every $R\cC$-module $M$ 
and every $\sigma \in s(\cC)$, $(\Res _{\Ini} M) (\sigma)= \Res _{\pi_0 } ( M(\sigma_0)) $.
\end{lemma} 

Combining the observations on this section, we obtain the following spectral sequence:

\begin{theorem}[{S\l omi\' nska \cite[Cor 1.13]{Slominska-Homotopy}}]\label{thm:SlomSS}
Let $\cC$ be an EI-category and $M$ be an $R\cC$-module. Then, for every $n \geq 0$, the assignment 
$[\sigma] \to H^n ( \Aut_{s(\cC)} (\sigma); \Res _{\pi_0} M( \sigma _0 ))$ defines an $R[s(\cC)]$-module, 
denoted by $\cA^n$, and there is a spectral sequence
$$ E_2 ^{p,q}= H^p ( [s(\cC)] ; \cA^q ) \Rightarrow H^{p+q}  ( \cC ;  M).$$ 
\end{theorem}

\begin{proof} By Proposition \ref{pro:SC-SourceRegular}, the category $s(\cC)$ is target regular.
Applying Theorem \ref{thm:RegularSS}(1) to $s(\cC)$ with coefficients in $\Res _{\Ini} M$, we 
obtain a spectral sequence $$ E_2 ^{p,q}= H^p ( [s(\cC)] ;  H^q ( \Aut_{s(\cC) } (-) ; (\Res_{\Ini} M)(-) ) 
\Rightarrow H^{p+q}  ( s( \cC) ;  \Res _{\Ini} M).$$ 
By Lemma \ref{lem:Modules}, for every $\sigma\in s(\cC)$, we can replace 
$(\Res _{\Ini} M) (\sigma)$ with $\Res _{\pi_0 } (M(\sigma_0))$.
Now  the result follows from the isomorphism in Proposition \ref{pro:SC-Cohomology}.
\end{proof}

This spectral sequence constructed in  \cite[Cor 1.13]{Slominska-Homotopy} has a
different expression for the $E_2$-term, but it is easy to see the $E_2$-term of the sequence 
in \cite[Cor 1.13]{Slominska-Homotopy} is isomorphic to the terms of the spectral sequence 
given above.  This spectral sequence with trivial coefficients is also constructed in \cite{Linckelmann-OnH}.


\section{Spectral sequences for the transporter category}\label{sect:Transporter}

Let $G$ be a discrete group and $\cH$ be a collection of subgroups of $G$ (closed under conjugation).  
The \emph{transporter category} $\cT_{\cH} (G)$ is the category whose objects are
the subgroups $H \in \cH$ and whose morphisms are given by 
$$\Mor _{\cT _{\cH} (G) } ( H, K) =N_G (H, K):=\{ g\in G \, | \, gHg^{-1} \leq K\}.$$
Composition of two morphisms is defined by group multiplication. We are interested in the following quotient categories
of the transporter category.

\begin{definition} 
Let $G$ be a discrete group and $\cH$ be a collection of subgroups of $G$.
\begin{enumerate}
\item The \emph{orbit category} $\cO _{\cH} (G)$ is the category 
whose objects are the subgroups $H \in \cH$ and whose 
morphisms are given by $$\Mor _{\cO _{\cH} (G) } ( H, K) =\{ Kg \, | \, g\in G, gHg^{-1} \leq K\}.$$
\item The \emph{fusion category} $\cF _{\cH} (G)$ is the category whose objects are the subgroups $H \in \cH$ and whose 
morphisms are given by $$\Mor _{\cF _{\cH} (G) } ( H, K) =\{ gC_G(H) \, | \, g\in G, gHg^{-1} \leq K\}.$$
\item The \emph{fusion orbit category} $\overline \cF _{\cH} (G))$ is the category 
whose objects are the subgroups $H \in \cH$ and whose morphisms 
are given by $$\Mor _{\overline \cF _{\cH} (G)} ( H, K) =\{ KgC_G(H) \, | \, g\in G, gHg^{-1} \leq K\}.$$
\end{enumerate}
\end{definition}

The morphisms in the orbit category can be identified with the $G$-maps $G/H \to G/K$ between 
the $G$-sets of the left cosets $G/H$ and $G/K$.  The morphisms in the fusion category can be viewed 
as group homomorphisms $c_g : H \to K$ induced by conjugation with $g \in G$. Note that the fusion 
orbit category is a quotient category of the fusion category. We have
$$\Mor _{\overline \cF _{\cH} (G)} ( H, K)\cong \Inn (K) \backslash \Mor _{\cF _{\cH} (G) } ( H, K) $$
where $\Inn (K)\cong K/ Z(K)$ is the group of conjugations $c_k : K \to K$ induced by elements $k \in K$. 
The relationship between these categories can be described by the following diagram where each arrow 
represents the projection functor to the corresponding quotient category.
$$\xymatrix{  & \cT _{\cH} (G) \ar[dl]_{\pi_1} \ar[dr]^{\pi_2}&  \\ 
\cO _{\cH} (G)  \ar[dr]_{\pi_3} & & \cF _{\cH} (G)  \ar[dl]^{\pi_4} \\
& \overline \cF_{\cH} (G) & \\} $$

The quotient maps $\pi_1$ and $\pi_2$ define extensions of categories and the corresponding LHS-spectral sequences 
for these extensions are well-known spectral sequences. They are related to the subgroup and centralizer decompositions
for the cohomology of the group $G$. 

\subsection{Subgroup decomposition} For every $H, K\in \cH$, the left $K$-action on $N_G(H, K)$ is free.
Hence the quotient map $\pi _1: \cT _{\cH } (G) \to \cO _{\cH } (G)$ defines a target regular extension
$$\cE:  \{ K \} _{K \in \cH } \to \cT_{\cH } (G) \xrightarrow{\pi_1} \cO _{\cH } (G).$$
Applying Theorem \ref{thm:XuSSCoh} to this extension, we obtain:

\begin{proposition}\label{pro:SubgroupDec} For every $R\cT _{\cH } (G)$-module $M$, there is a spectral sequence
$$ E_2 ^{p,q}= H^p ( \cO _{\cH} (G) ;  \cH^q _M ) \Rightarrow H^{p+q} ( \cT _{\cH} (G) ;  M)$$
where $\cH^q _M$ is the $R\cO_{\cH} (G)$-module such that $\cH ^q_M(K)=H^q(K; M(K))$ for every $K\in \cH$.
\end{proposition}
In this case, the $R\cO _{\cH} (G)$ module structure can be described as follows: Given a $G$-map
$f:G/H\to G/K$, let $g\in G$ such that $f(H)=g^{-1} K$. Then $gHg^{-1}\leq K$. The $\cT_{\cH} (G)$ action
on $\cK=\{ K\}_{K\in \cH}$ is given by $c_g : H \to K$ that sends $h\in H$ to $ghg^{-1}$ in $K$. Hence, 
$f$ induces a homomorphism 
$$H^* (K; M(K)) \xrightarrow{c_g^*}  H^*(H; \Res _{c_g }  M(K) ) \xrightarrow{M(f)_*} H^* (H; M(H)).$$  
This is the $R\cO_{\cH } (G)$-module structure for the assignment $K \to H^q (K; M(K))$.

Let $K_{\cH}$ denote the simplicial complex
whose simplices are chains in $\cH$ with the order relation given by inclusion of subgroups. 
The realization of the transporter category $|\cT_{\cH} (G)|$ is homotopy 
equivalent to the Borel construction $EG\times _G K_\cH$ (see \cite[Rem 2.2]{Grodal-Endo}). 
If we take the $R\cT_{\cH} (G)$-module $M$ to be the constant functor 
$A$, then we obtain a spectral sequence
 $$ E_2 ^{p,q}= \underset{\cO _{\cH} (G)}{\lim {}^p}  \ H^q (-; A)  \Rightarrow H^{p+q} ( EG \times _G K_{\cH} ;  A).$$
where $H^q(-; A)$ denotes the $q$-th group cohomology functor considered as an $R\cO_{\cH} (G)$-module.

\begin{definition}\label{ample} A collection of subgroups $\cH$ of a finite group $G$ is called \emph{ample} if 
$$H^* (EG \times _{G} K_{\cH} ; \bbF_p)\cong H^* (BG; \bbF_p).$$
\end{definition}

Any collection that includes the trivial subgroup is an ample collection, hence the collection of all $p$-subgroups
in $G$ is an ample collection. The collection of all nontrivial $p$-subgroups in a finite group 
$G$ is also ample when $p$ divides the order of $G$ (see Brown \cite[Sect X.7]{Brown-Book}).
A list of all ample collections of subgroups in a finite group 
can be found in \cite{GrodalSmith} and \cite{BensonSmith-Book}.

When $\cH$ is an ample collection, the spectral sequence above gives a spectral sequence of the form
$$ E_2 ^{p,q}= \underset{\cO _{\cH} (G)}{\lim {} ^p} \ H^q (-; \bbF_p)  \Rightarrow H^{p+q} (BG; \bbF_p).$$
This spectral sequence is usually constructed as the Bousfield-Kan spectral sequence of the mod-$p$ subgroup 
decomposition of $BG$ (see \cite[Sect 7]{Dwyer-Book}). For infinite groups with finite virtual cohomological dimension (vcd), a similar 
spectral sequence involving Farrell cohomology is constructed in \cite[Thm 4.6]{Lee-Subgroup}.


\subsection{Centralizer decomposition} For every $H, K\in \cH$, the right $C_G (H)$-action on 
$N_G(H, K)$ is free. Hence the functor $\pi _2: \cT _{\cH } (G) \to \cF _{\cH } (G)$ 
defines a source regular extension 
$$\cE:  \{ C_G(H) \} _{H \in \cH } \to \cT_{\cH } (G) \xrightarrow{\pi_2} \cF _{\cH } (G).$$
Applying Corollary \ref{cor:XuSSExt-Source} to this extension, we obtain:

\begin{proposition}\label{pro:CentralizerDec}
For every left $R \cT _{\cH } (G)$-module $M$, there is a spectral sequence
$$ E_2 ^{p,q}= H^p ( \cF _{\cH} (G) ;  \cH ^q _{C, M} ) \Rightarrow H^{p+q} ( \cT _{\cH} (G) ;  M)$$
where $\cH^q _{C, M}$ is the left $R \cF_{\cH} (G)$-module such that
$\cH ^q _{C, M} (H)=H^q ( C_G (H) ; M(H))$ for every $H \in \cH$.
\end{proposition}

The left $R\cF _{\cH} (G)$-module structure can be defined in a similar way to the orbit category case. The difference here is 
that since the extension is source regular, the action of $T_{\cH} (G)$ on $\cK=\{ C_G (H)\} _{H \in \cH}$ is given by
a contravariant functor, i.e. $c_g : H \to K$ induces a group homomorphism $C_G(K) \to C_G(H)$. So the induced action on cohomology
is given by a covariant functor.  

As in the subgroup decomposition case, if we take the left $R\cT_{\cC} (G)$-module $M$ to be the constant functor 
$\bbF_p$, and $\cH$ to be an ample collection, then we obtain a spectral sequence
 $$ E_2 ^{p,q}= \underset{\cF _{\cH} (G)}{\lim {}^p} \ H^q ( C_G (-);  \bbF_p)  \Rightarrow H^{p+q} ( BG;  \bbF_p).$$
This spectral sequence is the same spectral sequence as the Bousfield-Kan spectral sequence associated 
to the centralizer decomposition of $BG$ (see \cite[\S 7]{Dwyer-Book}). 
For infinite groups with finite vcd, a spectral sequence for the centralizer decomposition 
over the collection of all nontrivial elementary abelian $p$-subgroups is constructed in \cite{Lee-Centralizer}.


\subsection{Normalizer decomposition} 
 
The transporter category $\cT_{\cH} (G)$ is an EI-category, so we can apply 
Theorem \ref{thm:SlomSS} to $\cT_{\cH} (G)$ to obtain a spectral sequence 
$$E_2 ^{p.q} = H^p ([s(\cT_{\cH} (G))]; \cA_{\cT} ^q ) \Rightarrow H^{p+q} (\cT _{\cH } (G); M)$$ 
where $M$ is a right $\cT _{\cH } (G)$-module and $\cA_{\cT} ^q$ is the $R[s(\cT _{\cH } (G)]$-module such that  
$$ \cA_{\cT} ^q ([\sigma])=H^q (\Aut_{s(\cT _{\cH } (G)) } (\sigma) ; \Res_{\pi_0} M (\sigma_ 0 ) ).$$ 
Every strict chain $$ \sigma := ( H_0 \xrightarrow{\alpha_1} H_1 \xrightarrow{\alpha_2} \cdots  \xrightarrow{\alpha_n} H_n ) $$
in $\cT_{\cH } (G)$ is isomorphic to a strict chain where all the morphisms $\alpha_i$ are inclusion of subgroups.
Hence we can take the representatives of isomorphism classes of chains of morphisms to be 
the strictly increasing chains of subgroups $\sigma := (H_0 < H_1<  \cdots < H_n)$ in $\cH$. 
For such a chain $\sigma$, we have 
$$\Aut _{s(\cT_{\cH} (G))} (\sigma) =N_G (\sigma) :=\{ g \in G \, | \, gH_i g^{-1} \leq H_i \text{ for all } i \}.$$ 
Let  $S(\cH)$ be the poset category of strictly increasing chains in $\cH$ with order relation  
given by inclusion of subchains, and $s(\cH)$ denote its opposite category. 
The category of the $G$-orbits of chains in $\cH$, denoted by $s(\cH)/G$, is the category
whose objects are $G$-orbits of strict chains in $\cH$ and whose morphisms are defined in such a way that
there is a morphism $[\sigma]\to [\tau]$ if there is a $g\in G$ such that $\tau$ is subchain of $g\sigma$.
The poset category  $[s(\cT_{\cH} (G))]$ of isomorphism classes of chains in $s(\cT _{\cH} (G))$
can be identified with $s(\cH )/G$. Hence we conclude:

\begin{proposition}\label{pro:NormalizerDec}
For every $R \cT _{\cH } (G)$-module $M$, there is a spectral sequence of the form
$$E_2 ^{p.q} = H^p ( s( \cH)/G ; \cN^q) \Rightarrow H^{p+q} (\cT _{\cH } (G); M)$$
where $\cN^q$ is the $R(s(\cH)/G)$-module such that $\cN ^q([\sigma])= H^q ( N_G (\sigma) ; \Res_{\pi_0} M(\sigma _0) )$
for every $[\sigma]_G$ in $s(\cH)/G$.
\end{proposition}

As before, if we take $M$ to be the constant functor $\bbF_p$, and $\cH$ to be an ample collection, then we obtain a 
spectral sequence
$$E_2 ^{p.q} = \underset{s(\cH)/G}{\lim {}^p} \ H^q (N_G (-); \bbF_p ) \Rightarrow H^{p+q} (BG ; \bbF_p).$$
 This is the same spectral sequence as the Bousfield-Kan spectral sequence for the normalizer decomposition 
 of the classifying space of $G$ (see \cite[\S 7]{Dwyer-Book}). 

A collection of subgroups $\cH$ is called \emph{normalizer sharp} if in the above spectral sequence 
we have $E_2 ^{p,q}=0$ for all $p> 0$ and $q\geq 0$.
For finite groups, there are many normalizer sharp collections of $p$-subgroups 
(see \cite{GrodalSmith}). When the collection $\cH$ is normalizer sharp, this spectral sequence 
allows us to calculate the cohomology groups of $G$ in terms of the cohomology of the normalizers of chains of subgroups in $\cH$.

\subsection{Orbit and fusion orbit category} 
In general the projection functors $\pi_3 : \cO _{\cH} (G) \to \overline \cF _{\cH} (G)$ and $\pi_4 : \cF _{\cH} (G) \to \overline \cF _{\cH} (G)$ 
that are defined above are not regular functors as the following counterexamples illustrate:

\begin{example} For $\pi_3: \cO _{\cH} (G)\to \overline  \cF _{\cH} (G)$, let $\cH$ be any collection such that $1, K \in \cH$
with $K \neq 1$.
Then $\Mor _{\cO _{\cH} (G)} (1, K)\cong G/K$ and the centralizer $C_G (1)=G$ does not act freely on $G/K$.
This shows that $\pi_3$ is not source regular. For $\pi_4: \cF _{\cH} (G) \to \overline \cF _{\cH} (G)$, let
$\cH$ be a collection such that $1, K \in \cH$ where $K$ is a nonabelian group. 
Then $\Mor _{\cF _{\cH} (G)} (1, K) \cong 
G/C_G(1)\cong 1$ and  $\Inn (K) \cong K/ Z(K) \neq 1$, so $\Inn(K)$ does not act freely on $\Mor _{\cF _{\cH} (G)} (1, K)$.
Hence $\pi_4$ is not a target regular functor for the collection $\cH$.
\end{example}

In some special cases the functors $\pi_3$ and $\pi_4$ define regular extensions. One important special case is the following:
Let $G$ be a finite group. A $p$-subgroup $P \leq G$ is called \emph{$p$-centric} if $Z(P)$ is a Sylow $p$-subgroup 
of the centralizer $C_G(P)$. In this case we have $C_G(P)\cong Z(P) \times C'_G (P)$ where $C'_G (P)$ is a subgroup 
of $C_G(P)$ with order coprime to $p$. Let ${\mathcal Ce}$ denote the collection of $p$-centric subgroups in $G$. 
When $\cH={\mathcal Ce}$, we denote 
$\cO_{\cH} (G)$ by $\cO ^c (G)$ and $\overline \cF_{\cH} (G)$ by $\overline \cF ^c (G)$. 

\begin{lemma}\label{lem:SourceReg} The projection functor $\pi _3: \cO ^c (G) \to \overline \cF^c (G)$ defines a source regular 
extension
$$ \{ C'_G (P) \} _{P\in {\mathcal Ce} } \to \cO ^c (G) \maprt{\pi_3}\overline \cF ^c (G).$$
\end{lemma}

\begin{proof} For every $p$-centric subgroup $P \leq G$, we have
 $$\ker\{ (\pi_3)_{P,P} : \Aut _{\cO ^c (G)} (P) \to \Aut _{\overline \cF ^c(G) } (P) \}$$
is isomorphic to $C_G(P) P/P \cong C_G (P)/Z(P)\cong C'_G(P)$. Note that
$$\Mor _{\cO ^c (G)}  (P, Q) \cong Q\backslash N_G(P, Q)   \cong \{ Qg \, | \, g\in G, \, gPg^{-1} \leq Q\}.$$   
For every $c\in C'_G (P)$, the right action of $c$ on $\Mor _{\cO ^c (G)} (P, Q)$ is defined by 
$(Qg)\cdot c=Qgc$. If for some $Qg \in \Mor _{\cO ^c (G)} (P, Q)$ and $c\in C'_G (P)$, the equality $Qgc=Qg$ holds, then 
$gcg^{-1} \in Q$ which implies that $c=1$ because $c$ has order coprime to $p$. 
Therefore $C'_G (P)$ acts freely on 
$\Mor _{\cO ^c (G)}  (P, Q)$, hence $\pi_3$ defines a source regular 
extension.
\end{proof}

Applying the spectral sequence in Corollary \ref{cor:SSGY} to the extension in Lemma \ref{lem:SourceReg}, 
for every $R\overline  \cF _{\cH } (G)$-module $M$,
we obtain a spectral sequence
$$ E_2 ^{p,q}= \Ext^p _{R \overline \cF ^c (G) } ( H_q ( C'_G(-); R) , M ) 
\Rightarrow H^{p+q} (\cO ^c (G) ; \Res _{\pi} M).$$ 
If $R$ is a $p$-local ring, i.e., all the primes other than $p$ are invertible in $R$, 
then $$H_q(C' _G (P); R)\cong \begin{cases} R & \text{if  } \ q=0 \\ 0 &\text{if  } \ q>0.\end{cases}$$   
In this case, the spectral sequence collapses at the $E_2$-page to the horizontal line at $q=0$.
As a consequence we obtain the following:

\begin{proposition}[{Broto-Levi-Oliver \cite[Lem 1.3]{BLO1}}] 
Let $G$ be a finite group and $R$ be a $p$-local ring. Then for every 
$R\overline  \cF _{\cH } (G)$-module $M$, there is an isomorphism
$$H^* (\overline \cF ^c  (G) ; M) \cong H^* (\cO ^c (G); \Res _{\pi_3} M)$$
induced by the functor $\pi_3 : \cO _{\cH } (G) \to \overline \cF _{\cH} (G)$.
\end{proposition}
  
When $G$ is an infinite discrete group that realizes 
a fusion system $\cF$,  there is a similar extension of categories and 
a spectral sequence. In this case the kernel is given by a signalizer functor and 
one obtains a spectral sequence with two horizontal lines
(see \cite[\S 7]{GundoganYalcin}).


\section{Spectral sequences for linking systems}\label{sect:Linking}

Let $G$ be a finite group and $S$ be a Sylow $p$-subgroup of $G$. 
The \emph{centric linking system} $\cL ^c _S (G)$ is the category whose objects are 
the $p$-centric subgroups $P$ of $G$ that lies in $S$, and whose morphisms are given by 
$$\Mor _{\cL ^c _{S} (G)} (P, Q) =\{ g C'_G (P) \, |\, g \in G, gPg^{-1} \leq Q\}$$
where $C'_G (P)$ is the subgroup of $C_G(P)$ of order coprime to $p$ such that $C_G (P) \cong Z(P) \times C'_G (P)$.
From this definition, it is easy to see that there is a source regular extension of categories
$$\cE: \{ C' _G (P) \} \xrightarrow{i} \cT _S ^c (G) \xrightarrow{\pi} \cL _S ^c (G)$$
where $\cT^c _S(G)$ is the transporter category whose objects are the $p$-subgroups $P\leq S$ that are $p$-centric, 
and whose morphisms $P\to Q$ are given by $N_G (P, Q)$.
By Corollary \ref{cor:SSGY}, for every $R \cL _S ^c (G)$-module $M$, there is a spectral sequence
$$ E_2 ^{p,q}= \Ext^p _{R \cL^c  _{S} (G) } ( H_q ( C'_G(-); R) , M ) \Rightarrow H^{p+q} (\cT ^c_{S} (G) ; \Res _{\pi} M).$$ 
Using this spectral sequence, we obtain the following:

\begin{proposition}[{\cite[Prop 1.1]{BLO1}}] Let $R$ be a $p$-local ring. Then for every 
$R\cL ^c _{S} (G)$-module $M$, there is an isomorphism
$$H^* (\cL ^c _S (G) ; M) \cong H^* (\cT ^c  _{S} (G); \Res _{\pi} M).$$
\end{proposition}

\begin{proof} This follows from the fact that if $R$ is $p$-local, then $H_q(C' _G (P); R)\cong R$ if  $q=0$, and  
$H_q(C' _G (P); R)=0$ if  $q>0$.
\end{proof} 

For a finite group $G$ with Sylow $p$-subgroup $S$, \emph{the fusion system of $G$ defined on $S$} 
is the category $\cF_S(G)$ whose objects 
are the subgroups of $S$ and whose morphisms $P\to Q$ are given by 
the group homomorphisms $c_g: P\to Q$ defined by conjugation with an element $g\in G$. 
Let ${\mathcal Ce}_S$ denote the collection of $p$-centric subgroups of $G$ that lies in $S$ 
and $\cF_S^c (G)$ denote the subcategory whose object set is the collection ${\mathcal Ce}_S$.
From the definitions, one can easily see that the quotient functor $\cL ^c _S(G) \to \cF^c_{S} (G)$ 
defines a source regular extension
\begin{equation}\label{eqn:ExtLinkingFinite}
\cE:  \{ Z(P) \} _{P \in {\mathcal Ce}_S} \to \cL ^c _S (G) \to \cF_S ^c (G).
\end{equation}
There is a spectral sequence associated to this source regular extension. In the next section, we discuss this extension 
in more general context of $p$-local finite groups.


\subsection{$p$-Local finite groups}

A \emph{fusion system} $\cF$ over a $p$-group $S$ is a category whose objects are subgroups of $S$ 
and whose morphisms $\Mor _{\cF} (P, Q)$ are injective group homomorphisms $P\to Q$ satisfying certain axioms 
(see \cite{AKO-Book} for a definition). A fusion system is saturated if it satisfies some additional axioms. In this paper we
assume all fusion systems are saturated. The fusion system $\cF_S(G)$ of a finite group $G$ is an 
example of a fusion system. A subgroup $P \leq S$ is called \emph{$\cF$-centric} if  $C_S(Q)\leq Q$ 
for every $Q\leq S$ which is $\cF$-isomorphic to $P$. The full subcategory of $\cF$ 
whose objects are the $\cF$-centric subgroups of $S$ is denoted by $\cF^c$.

\begin{definition}[{\cite[Def 3.2]{AshOliver}}]\label{def:LinkingSystem}
Let $\cF$ be a fusion system over the finite $p$-group $S$. A \emph{centric linking system}
associated to $\cF$ is the category $\cL$ whose objects are the $\cF$-centric subgroups 
of $S$, together with a pair of functors $$\cT _S ^{c} (S) \xrightarrow{\epsilon}  \cL \xrightarrow{\pi} \cF^c$$
satisfying the following conditions:
\begin{enumerate}
\item The functors $\epsilon$ and $\pi$ are the identity on objects and $\epsilon _{P, Q} : N_{S}(P, Q) \to \Mor _{\cL} (P, Q)$ is injective
and $\pi _{P, Q} : \Mor _{\cL} (P, Q) \to \Mor _{\cF} (P, Q)$ is surjective.
\item For each $P, Q \in \Ob(\cL)$, $\epsilon _{P, P} (P) \leq \Aut _{\cL} (P)$ acts freely on the set 
$\Mor _{\cL } (P, Q)$ by precomposition and $\pi _{P, Q}$ induces a bijection from 
$\Mor _{\cL } (P, Q)/ \epsilon_{P,P} (Z(P))$ onto $\Mor _{\cF ^c} (P, Q)$.
\item For each $P, Q \in \Ob (\cL)$ and $g \in N_S (P, Q)$, the composite functor $\pi \circ \epsilon$ 
sends $g \in N_S(P, Q)$ to $c_g \in \Mor _{\cF} (P, Q)$.
\item For $P, Q \in \Ob (\cL)$, $\psi \in \Mor _{\cL} (P, Q)$, and $g \in P$, the following square commutes in $\cL$.
$$\xymatrix{P \ar[r]^{\psi} \ar[d]_{\epsilon _{P,P} (g)} & Q \ar[d]^{\epsilon _{Q,Q} (\pi (\psi ) (g)) } \\ P \ar[r]^{\psi} & Q.}$$ \qed
\end{enumerate}
\end{definition} 

For every fusion system $\cF$, there is a unique associated centric linking system $\cL$ (see \cite{Chermak}). 
The triple $(S, \cF, \cL)$ is called a \emph{$p$-local finite group}. For every finite group $G$ with Sylow $p$-subgroup $S$, 
the triple $(S, \cF_S(G), \cL_S ^c (G))$ is a $p$-local finite group (see \cite[\S 1]{BLO2}). 
By condition (2) in Definition \ref{def:LinkingSystem}, we can conclude the following:

\begin{lemma}\label{lem:ExtLinkingSource}
For every $p$-local finite group $(S, \cF, \cL)$, 
the projection map $\pi : \cL \to \cF^c$ defines a source regular extension
 $$\cE:  \{ Z(P) \}_{P \in \cF^c} \maprt{\delta} \cL  \maprt{\pi} \cF^c $$
where $\delta_P : Z(P) \to \Aut _{\cL } (P)$ is the homomorphism defined as the restriction 
of the homomorphism $\epsilon_{P,P}: N_S(P) \to \Aut _{\cL } (P)$ to $Z(P)$.  
\end{lemma}

Note that this is the $p$-local finite group version of the extension given in \ref{eqn:ExtLinkingFinite}.
This extension is one of the examples given by Xu in \cite[Example 4.1.1]{Xu-CohSmall} for source regular extensions of categories.
Applying Theorem \ref{thm:RegularSS} to this extension gives two spectral sequences,
one with coefficients in a left $R\cL$-module, one with (right) $R\cF^c$-modules. Here we state the one with $R\cF^c$-modules:

\begin{proposition}\label{pro:Z(P)SS}
Let $(S, \cF, \cL)$ be a $p$-local finite group. For every $R\cF^c$-module $M$, there is 
a spectral sequence
$$ E_2 ^{p,q}= \Ext _{R\cF^c} ^p ( H_q (Z(-); R), M)  \Rightarrow H^{p+q} (\cL ; \Res _{\pi} M)$$ 
where $H_q (Z(-); R)$ denotes the $R\cF^c$-module such that $P \to H_q (Z(P); R)$ for every $P\in \cF^c$.
\end{proposition}

The $R\cF ^c$-module structure on $H_q (Z(P); R)$ is defined by the $\cL$-action on the kernel $\cK=\{ Z(P)\}$.
Note that for every $\varphi: P\to Q$ in $\cF^c$, there is a morphism $\widetilde \varphi: P\to Q$ in $\cL$
and it induces a group homomorphism $Z(\widetilde \varphi): Z(Q)\to Z(P)$ since the above extension is source regular.
The induced map $Z(\widetilde \varphi)_*: H_q (Z(Q); R) \to H_q (Z(P); R)$ defines the (right) $R\cF^c$-module structure used 
in the Proposition \ref{pro:Z(P)SS}.

\subsection{Subgroup decomposition for $\cL$} 
 
Let $(S, \cF, \cL)$ be a $p$-local finite group. The \emph{orbit category of $\cF$} is the category $\cO (\cF)$
whose objects are subgroups of $S$ and the morphisms are given by 
$$\Mor _{\cO (\cF) } (P,Q) :=\Inn (Q) \backslash \Mor _{\cF} (P, Q).$$
We denote by $\cO ^c (\cF)$ 
the full subcategory of $\cO (\cF)$ whose objects are the $\cF$-centric subgroups of $S$. 
There is a quotient functor $\pi': \cF ^c \to \cO ^c (\cF)$ which sends a morphism $P\to Q$ in $\cF^c$ to its orbit under $\Inn (Q)$-action. 
Let $\widetilde \pi: \cL \to \cO ^c (\cF)$ denote
the composition
$$\widetilde\pi : \cL \maprt{\pi} \cF ^c \maprt{\pi'} \cO ^c (\cF).$$

\begin{lemma}\label{lemma:ExtLinkingTarget}
For every $p$-local finite group $(S, \cF, \cL)$, 
the functor $\widetilde \pi : \cL \to \cO(\cF^c)$ defines a target regular extension
$$\cE:  \{ P \}_{P \in \cF^c} \maprt{j} \cL  \maprt{\widetilde \pi} \cO ^c (\cF)$$
where for every $P\in \cF^c$, the homomorphism $j_P : P \to \Aut _{\cL } (P)$ is the restriction of $\epsilon_{P,P}: N_S(P) \to \Aut _{\cL } (P)$
to the subgroup $P$.  
\end{lemma}

\begin{proof} In $\cL$, every morphism is a monomorphism and an epimorphism in the categorical sense 
(see  \cite[Remark 2.10 and Prop 2.11]{Libman-NormDec}). So, the left $j_Q (Q)$ action on $\Mor _{\cL } (P, Q)$ is free.
Now we will show that the projection functor $\widetilde \pi : \cL \to \cO ^c (\cF)$ 
induces a bijection $$j_Q (Q) \backslash \Mor _{\cL} (P, Q) \cong \Mor _{\cO ^c (\cF)} (P, Q).$$
Suppose that $\varphi_1, \varphi_2 \in \Mor _{\cL} (P, Q)$ are such that $\varphi _2=j_Q (g)\varphi_1$ for some $g \in Q$. 
By (3) in Definition \ref{def:RegularExt}, $\pi (j_Q(g))= c_g$ in $\Aut _{\cF} (Q)$. This gives $$\widetilde \pi _{P,Q} (\varphi _2)=\pi' (c_g) \widetilde \pi _{P,Q} (\varphi_1)=\widetilde \pi _{P,Q} (\varphi _1)$$
since $\Mor _{\cO ^c (\cF) } (P,Q) :=\Inn (Q) \backslash \Mor _{\cF} (P, Q).$ So the map 
$$j_Q (Q) \backslash \Mor _{\cL} (P, Q) \to  \Mor _{\cO^c  (\cF)} (P, Q)$$
is well-defined. It is surjective because both $\pi$ and $\pi'$ are surjective on morphisms.
For injectivity assume that $\varphi_1, \varphi_2 \in \Mor _{\cL} (P, Q)$ are such that 
$\widetilde \pi (\varphi_1)=\widetilde \pi (\varphi )$. then there is a $x\in Q$ such that $\pi (\varphi _2) = c_x \circ \pi (\varphi _1)$ in $\Mor _{\cF } (P,Q)$.
Then by \cite[Lem 1.10(b)]{BLO2}, there is a unique $g \in Q$ such that $\varphi_2=j_Q(g) \varphi _1$ in $\Mor _{\cL } (P, Q)$. This proves the bijection hence the functor $\widetilde \pi$ is target regular.
\end{proof}

Applying Theorem \ref{thm:RegularSS} to this extension gives the following spectral sequence
(see \cite[Lem 4.2]{OliverVentura}):

\begin{proposition}\label{pro:SubgroupSS}
Let $(S, \cF, \cL)$ be a $p$-local finite group. Then, for every $R\cL$-module $M$, there is 
a spectral sequence
$$ E_2 ^{p,q}= H^p (\cO ^c (\cF)  ;  \cH^q_M )  \Rightarrow H^{p+q} (\cL ; M))$$ 
where $\cH^q _M$ denotes the $R\cO ^c (\cF)$-module such that  for every $P\in \cF^c$, 
we have $\cH^q_M (P)=H^q (P; \Res _{j_P} M)$.
For every morphism $\varphi : P\to Q$ in $\cO^c (\cF)$, the homomorphism 
$\cH ^q _M (\varphi) : \cH ^q _M (Q)\to \cH ^q _M (P)$ is defined by the composition
$$H^* (Q ; \Res _{j_Q} M) \xrightarrow{\widetilde \varphi^*} 
H^* (P; \Res _{\widetilde \varphi} \Res _{j_Q} M ) \xrightarrow{M(\widetilde {\widetilde \varphi})_* }  
H^* (P;  \Res _{j _P} M)$$ where $\widetilde \varphi : P \to Q$ 
be a morphism in $\cF^c$ lifting $\varphi$, and $\widetilde {\widetilde \varphi}$ be the morphism 
in $\cL$ lifting $\widetilde \varphi$.
\end{proposition}

\begin{proof} The spectral sequence follows from Lemma \ref{lemma:ExtLinkingTarget}
and Theorem \ref{thm:RegularSS}(1). The $R\cO ^c (\cF)$-module structure on $\cH ^q _M$ 
is the one induced by the $\cO^c (\cF)$-action on the kernel $\cK=\{P\}_{P\in \cF^c}$. We need to show that this module structure coincides 
with the one given above.  

For every $\varphi: P \to Q$ in $\cO ^c (\cF)$, let  $\widetilde \varphi$ and $\widetilde{\widetilde \varphi}$ are the lifts of $\varphi$
as in the proposition. The group homomorphism $K(\widetilde \varphi): P\to Q$ 
defined by $K(\widetilde \varphi ) (g)=g'$ for every $g\in P$ where $g'\in Q$ is the unique element such 
that $j_Q(g')  \circ \widetilde \varphi=\widetilde \varphi \circ j_Q(g)$. By condition (4) in 
Definition \ref{def:LinkingSystem}, we have $g'=\varphi (g)$. So the induced action of $K(\widetilde \varphi)$ on the cohomology
can be described by the composition given in the proposition.
\end{proof}

This spectral sequence also follows from the subgroup decomposition for the
linking system $\cL$ (see \cite[Prop 2.2]{BLO2}). For nontrivial coefficients this spectral sequence is a special 
case of the spectral sequence given for transporter categories (see \cite[Lem 4.2]{OliverVentura}).
There are two open problems related to the spectral sequence given in Proposition \ref{pro:SubgroupSS}. 
 
\begin{conjecture}[Sharpness Problem \cite{DiazPark}]
Let $(S, \cF, \cL)$ be a $p$-local finite group. Then for every $p>0$ and $q\geq 0$, 
$$E_2 ^{p,q} = \underset{\cO ^c (\cF )}{\lim {}^p}  \ H^q(-; \bbF_p)=0.$$
\end{conjecture}
 
For fusion systems realized by a finite group, the statement of the conjecture is proved 
by Diaz and Park \cite{DiazPark}. For a reduction of this problem to fusion systems 
with nontrivial center see \cite{Yalcin-Sharp}.

The second problem related to subgroup decomposition is about cohomology of linking systems 
with nontrivial local coefficient systems. For details of this problem we refer the reader to \cite{Molinier-Twisted}
and \cite{Molinier-Solvable}.


\subsection{Normalizer decomposition for $p$-local finite groups}

Let $(S, \cF, \cL)$ be a $p$-local finite group. The centric linking system $\cL$ is an EI-category, so we can apply  
Theorem \ref{thm:SlomSS} to $\cL$. For every $R\cL$-module $M$,
we obtain a spectral sequence of the form
\begin{equation}\label{eqn:NormSSLinking}
E_2 ^{p.q} = H^p ([s(\cL)]; \cA_{\cL} ^q ) \Rightarrow H^{p+q} (\cL ; M)
\end{equation}
where  $\cA_{\cL} ^q$ is the $R[s(\cL)]$-module such that $$\cA ^q _{\cL} ([\sigma])=
H^q (\Aut _{s(\cL) } (\sigma); \Res _{\pi_0} M (\sigma_0) )$$ for every $[\sigma] \in [s(\cL)]$. 
The poset category $[s(\cL)]$ and the functor $\pi_0: s(\cL)\to \cL$ are as defined in 
Section \ref{sect:RegularEI} for any EI-category.

Let $sd(\cF^c)$ denote the poset category of all chains $\sigma =(P_0 < P_1 < \cdots < P_n)$  with $P_i \in \cF^c$ 
where there is a unique morphism $\sigma \to \tau$ in $sd(\cF^c)$ if $\tau$ is a subchain 
of $\sigma$.  
Two chains $\sigma =(P_0 < P_1 < \cdots < P_n)$ and
$\tau =(Q_0 < Q_1 < \cdots < Q_n)$ are \emph{$\cF$-conjugate} if there is an isomorphism 
$\varphi : P_n \to Q_n$ in $\cF$ such that $\varphi (P_i)=Q_i$ for all $i\in [n]=\{0, \dots, n\}$. 
 
\begin{definition}\label{def:CatConjClasses} The category $\overline{s}d(\cF^c)$ of \emph{$\cF$-conjugacy 
classes of chains in $\cF ^c$} is the poset category 
whose objects are $\cF$-conjugacy classes $[\sigma]$  of chains $\sigma$ in $sd(\cF^c)$ 
where there is a unique morphism $[\sigma] \to [\tau]$ in $\overline{s}d(\cF ^c)$  if $\tau$ is $\cF$-conjugate to a subchain 
of $\sigma$.  
\end{definition}

We have the following:

\begin{lemma}[{\cite[\S 5]{Libman-NormDec}}] 
The poset categories $[s(\cL )]$ and $\overline{s}d(\cF^c)$ are isomorphic.
\end{lemma}

\begin{proof}
For each $P\leq Q$ in $S$, let $\iota _P  ^Q$ denote the image of the inclusion map 
$P \hookrightarrow Q$ under the functor $\epsilon : \cT ^c _S (S) \to \cL$. 
Every morphism $\varphi : P \to Q$ in $\cL$ factors as an isomorphism $\varphi': P \to P'$ 
in $\cL$ followed by a morphism 
 $\iota _{P'} ^Q$ with $P'=\pi (\varphi) (P)$ (see \cite[2.9]{Libman-NormDec}).
It follows that every strict chain 
$$ \sigma :=( \sigma_ 0 \xrightarrow{\alpha_1} \sigma_1  \xrightarrow{\alpha_2} 
\cdots  \xrightarrow{} \sigma_{n-1} \xrightarrow{\alpha_n} \sigma_n)$$
in $s(\cL)$ is isomorphic to a chain of the form
 $$ P_0 \xrightarrow{\iota_{P_0} ^{P_1}} P_1 \xrightarrow{\iota _{P_1} ^{P_2} } \cdots  \xrightarrow{} P_{n-1}
 \xrightarrow{\iota _{P_{n-1}} ^{P_n}} P_n$$ with $P_i\neq P_{i+1}$. 
 We can take representatives of $\cL$-isomorphism classes of chains of morphisms in $\cL$ to be 
the chains of this form. This shows that these poset categories $[s(\cL)]$ and $\overline{s}d(\cF^c)$ are isomorphic.
\end{proof}

For each chain $\sigma=(P_0<\dots < P_n)$ of $\cF$-centric subgroups of $S$, let $\Aut _{\cL } (\sigma)$ be the subgroup of $
\prod _{i=1} ^n  \Aut _{\cL} (P_i)$ formed by tuples $(\alpha _0, \dots, \alpha _n )$ of automorphisms 
$\alpha _i \in \Aut _{\cL } ( P_i)$ such that 
the following diagram commutes
$$\xymatrix{P_0   \ar[r]^{i^{P_1}_{P_0}}  \ar[d]^{\alpha_0}  &  P_1 \ar[r]^{i^{P_2} _{P_1}} \ar[d]^{\alpha_1}  
& \cdots  \ar[r] &  P_{n-1} \ar[d]^{\alpha_{n-1}}     \ar[r]^{i^{P_n} _{P_{n-1}}}  & P_n \ar[d]^{\alpha_n} \\  
P_0   \ar[r]_{i^{P_1}_{P_0}}  &  P_1 \ar[r]_{i^{P_2} _{P_1}} & \cdots  \ar[r] & P_{n-1} \ar[r]_{i ^{P_n} _{P_{n-1}} }  & P_n}$$
(see \cite[Def 1.4]{Libman-NormDec} for details). For each $j \in [n]$, the group homomorphism
$$\pi _j: \Aut _{\cL } (\sigma )\to \Aut _{\cL} (P_j )$$ is defined by $(\alpha_0, \dots, \alpha_n )\to \alpha_j$. 
With these definitions, the spectral
sequence in \ref{eqn:NormSSLinking} gives a spectral sequence stated in the following theorem.

\begin{theorem}\label{thm:NormDecLinking}
Let $(S, \cF, \cL)$ be a $p$-local finite group. For every $R\cL$-module $M$,
there is a spectral sequence 
$$E_2 ^{p.q} = H^p ( \overline{s}d(\cF^c) ; \cA_{\cL} ^q) \Rightarrow H^{p+q} (\cL ; M)$$ 
where $\cA_{\cL} ^q$ is the $R\overline{s}d(\cF^c)$-module such that $\cA^q ([\sigma])= H^q (\Aut _{\cL } (\sigma); \Res _{\pi_0} M(P_0))$
for every $[\sigma]=[P_0<\cdots<P_n]$ in $\overline{s}d(\cF^c)$.
\end{theorem}

This spectral sequence can also be obtained as the Bousfield-Kan spectral 
sequence for the normalizer decomposition of $|\cL|$ proved by Libman in \cite[Thm A]{Libman-NormDec}.

\begin{definition} The normalizer decomposition for $\cL$ over the collection of $\cF$-centric subgroups is \emph{sharp}
if for $M=\bbF_p$, we have 
$$E_2 ^{p,q}=\underset{\overline{s}d (\cF^c)}{\lim {}^p} \  H^q (\Aut _{\cL} (-); \bbF_p) =0$$ for all $p>0$ and $q\geq 0$. 
\end{definition}

It is an open problem whether or not the normalizer decomposition is sharp. It is proved in \cite[Thm 1.9]{Yalcin-Sharp} that
the normalizer decomposition for $\cL$ is sharp if and only if the subgroups decomposition for $\cL$ is sharp.


\subsection{Normalizer decomposition of $\cO ^c (\cF)$}

Let $\cF$ be a fusion system over $S$ and let $\cO ^c (\cF)$ denote the orbit category of $\cF^c$.
If we apply  Theorem \ref{thm:SlomSS} to $\cO ^c (\cF)$, we obtain a spectral sequence
\begin{equation}\label{eqn:OrbitCatSS}
E_2 ^{p.q} = H^p ([s( \cO ^c (\cF))]; \cA_{\cO} ^q ) 
\Rightarrow H^{p+q} (\cO ^c (\cF) ; M)
\end{equation}
where $M$ is an $R\cO (\cF^c)$-module and 
$\cA^q_{\cO}$ is the $[s(\cO ^c (\cF ) ]$-module such that
$$\cA_{\cO } ^q ([\sigma])=H^q (\Aut _{ s(\cO ^c (\cF ) )} (\sigma); \Res _{\pi_0} M(\sigma_0))$$
for every $[\sigma ] \in [s (\cO (\cF^c) ) ]$. 
We have the following observation:

\begin{lemma}\label{lem:Last} The poset category $[s(\cO^c (\cF) )]$ is isomorphic to the category $\overline{s}d(\cF^c )$.
\end{lemma}

\begin{proof} Let $\varphi : P\to Q$ be a morphism in $\cO^c (\cF)$ and $\widetilde \varphi : P\to Q$ be a lifting
of $\varphi$ in $\cF^c$ under the quotient functor $\pi': \cF^c \to \cO ^c (\cF)$. The morphism $\widetilde \varphi: P \to Q$  
can be written as a composition of an isomorphism
$\widetilde \varphi': P \to \widetilde \varphi (P)$ and an inclusion map $inc: \widetilde \varphi (P) \hookrightarrow Q$. Hence $\varphi: P \to Q$ 
can be written as a composition of $\pi' (\widetilde \varphi ') $ and $\pi' (inc)$. Moreover a morphism in $\cO ^c (\cF)$ is an isomorphism 
if and only if its lifting in $\cF^c$ is an isomorphism. This shows that every chain of morphisms $P_0 \maprt{\varphi _1} P_1 \to \cdots \maprt{\varphi _n}  P_n$ in $s(\cO^c (\cF))$ is isomorphic to a chain where all the maps $\varphi_i$ are of the images of inclusion maps under $\pi'$. This gives an isomorphism between poset categories $s(\cO^c (\cF) )$ and $\overline{s}d (\cF^c )$.
\end{proof}

Combining \ref{eqn:OrbitCatSS} with Lemma \ref{lem:Last} we obtain
 the following spectral sequence:

\begin{theorem}\label{thm:OrbitFusion}
Let $(S, \cF, \cL)$ be a $p$-local finite group. Then for every right $R\cO (\cF^c)$-module $M$, there is a spectral sequence
$$E_2 ^{p.q} = H ^p ( \overline{s}d (\cF^c) ; \cA_{\cO} ^q) 
\Rightarrow H^{p+q} (\cO (\cF ^c) ; M)$$ 
where $\cA_{\cO} ^q$ is the left $R\overline{s}d(\cF^c)$-module such that
 $\cA_{\cO} ^q ([\sigma])= H^q (\Aut _{s(\cO (\cF^c))} (\sigma); \Res _{\pi_0} M(P_0))$ for every $[\sigma]$ in $\overline{s}d (\cF ^c)$. 
\end{theorem} 
 
This spectral sequence has some similarity to the spectral sequence 
constructed in Theorem \cite[Thm 1.7]{Yalcin-Sharp}, but in general the $E_2$-terms of these two 
spectral sequences are different.

\section*{Conflict of interest}

The author declare that they have no conflict of interest.



\begin{thebibliography}{99}
%
\bibitem{AKO-Book}
M.~Aschbacher, R.~Kessar, and B.~Oliver, \emph{Fusion Systems in Algebra and Topology},
London Mathematical Society Lecture Note Series, 391. Cambridge Univ. Press, Cambridge, 2011.
%
\bibitem{AshOliver}
M.~Aschbacher and B.~Oliver, \emph{Fusion Systems}, Bulletin of the Amer. Math. Soc. \textbf{53} (2016), 555-615.
%
\bibitem{BauesWirsching}
H.~J.~Baues and G.~Wirsching, \emph{Cohomology of small categories}, J. Pure Appl. Algebra 38 (1985), 187--211.
%
\bibitem{BensonSmith-Book}
D.~J.~Benson and S.~D.~Smith, \emph{Classifying Spaces of Sporadic Groups}, volume 147 of Mathematical Surveys and Monographs, American Mathematical Society, Providence, RI, 2008.
%
\bibitem{BLO1}
C.~Broto, R.~Levi, and B.~Oliver, \emph{Homotopy equivalences of $p$-completed classifying spaces of finite groups}, Invent. Math. \textbf{151} (2003), 611-664.
%
\bibitem{BLO2}
C.~Broto, R.~Levi, and B.~Oliver, \emph{The homotopy theory of fusion systems}, J. Amer. Math. Soc. \textbf{16} (2003), 779-856.
%
\bibitem{Brown-Book}
K.~S.~Brown, \emph{Cohomology of Groups}, Graduate Texts in Mathematics, Springer, 1982.
%
\bibitem{Chermak}
A.~Chermak, \emph{Fusion systems and localities}, Acta Math. \textbf{211} (2013), no. 1, 47-139. 
%
\bibitem{DiazPark} 
A.~Diaz and S.~Park, \emph{Mackey functors and sharpness for fusion systems}, Homology, Homotopy and Applications, \textbf{17} (2015), 147-164
%
\bibitem{DjamentTouze-Functor}
A.~Djament and A.~Touze, \emph{Functor homology over an additive category}, preprint 2021. (ArXiv:2111.09719).

\bibitem{Dwyer-Book}
W.~G.~Dwyer and H.-W.~Henn, \emph{Homotopy theoretic methods in group cohomology}, 
Advanced Courses in Mathematics--CRM Barcelona, Birkhauser Verlag, Basel, 2001.
%
\bibitem{GZ-Book} P.~Gabriel and M.~Zisman, \emph{Calculus of fractions and homotopy theory}, 
Ergebnisse der Mathematik und ihrer Grenzgebiete, Band 35 Springer-Verlag New York, Inc., New York 1967.
%
\bibitem{GNT} 
Galvez-Carrillo, Imma, Frank Neumann, and Andrew Tonks. \emph{Gabriel–Zisman Cohomology and Spectral Sequences}
Applied Categorical Structures \textbf{29} (2021), 69-94.
%
\bibitem{Grodal-Endo} 
J.~Grodal, \emph{Endotrivial modules for finite groups via homotopy theory}, 
Journal of the American Mathematical Society 36 (1), 177-250, 2023.
%
\bibitem{GrodalSmith} 
J.~Grodal and S.~D.~Smith, \emph{Propagating sharp group homology decompositions}, Adv. in Math. \textbf{200} (2006), 525-538.
%
\bibitem{GundoganYalcin}
M.~S.~G\" undo\u gan and E.~Yal\c c\i n, \emph{Cohomology of infinite groups realizing fusion systems}, 
Journal of Homotopy and Related Structures \textbf{14} (2019), 1103--1130.
%
\bibitem{HPY}
I.~Hambleton, S.~Pamuk, and E.~Yal{\c{c}}{\i}n, \emph{Equivariant
{CW}-complexes and the orbit category}, Comment. Math. Helv. \textbf{88}
(2013), 369--425.
%
\bibitem{Hoff-Extensions}
G.~Hoff, \emph{Cohomologies et extensions de categories}, 
Mathematica Scandinavica, 1994, 191--207.
%
\bibitem{Huseinov}
A.~A.~Husainov, \emph{Homological dimension theory of small categories}, Journal of Mathematical Sciences, Vol. 110, No. 1, (2002), 2273--2321.
%
\bibitem{Jackowski-Transfer}
S.~Jackowski, \emph{A transfer map in the cohomology of small categories}, Bulletin of the Polish Academy of Sciences, \textbf{35} (1987), 161-166.

%
\bibitem{JM-Elementary}
S.~Jackowski and J.~McClure, \emph{Homotopy decomposition of classifying spaces via elementary abelian subgroups}, Topology \textbf{31} (1992) 113--132.
%
\bibitem{JMO}
S.~Jackowski, J.~McClure, and B.~Oliver, \emph{Homotopy classification of self-maps of BG via G-actions}, 
Annals of Math. \textbf{135} (1992), 183–270.
%
\bibitem{JackowskiSlominska}
S.~Jackowski and J.~S\l omi\' nska, \emph{$G$-functors, $G$-posets and homotopy decompositions of $G$-spaces}, 
Fund. Math. \textbf{169} (2001) 249--287.
%
\bibitem{KashiwaraSchapira-Book}
M.~Kashiwara and P.~Schapira, \emph{Categories and Sheaves}, A Series in Comprehensive Studies in Mathematics, 
Vol. 332, Springer (2006).
%
\bibitem{Lee-Centralizer} 
C.N.~Lee, \emph{Farrell cohomology and centralizers of elementary abelian $p$-subgroups},
Mathematical Proceedings of the Cambridge Philosophical Society. Vol. 119. No. 3. 
Cambridge University Press, 1996.
%
\bibitem{Lee-Subgroup}
C-N.~Lee, \emph{A homotopy decomposition for the classifying space of virtually torsion-free groups and applications},
Math. Proc. Camb. Phil. Soc. (1966), \textbf{120}, 663. 
%
\bibitem{LeviRagnarsson}
R.~Levi and K.~Ragnarsson, \emph{$p$-local finite group cohomology}, 
Homology Homotopy Appl. \textbf{13} (2011), 223--257.
%
\bibitem{Libman-NormDec}
A.~Libman, \emph{The normalizer decomposition for $p$-local finite groups}, Alg. Geom. Topology \textbf{6} (2006), 1267-1288. 
%
\bibitem{Linckelmann-Orbit} 
M.~Linckelmann, \emph{The orbit space of a fusion system is contractible}, Proc. London Math. Soc. (3) \textbf{98} (2009), 191--216.
%
\bibitem{Linckelmann-OnH}
M.~Linckelmann, \emph{On $H^*(\cC; k^{\times})$ for fusion systems}, Homology Homotopy Appl. \textbf{11} (2009), 203--218.
%
\bibitem{Lueck-Book}
W.~L{\"u}ck, \emph{Transformation groups and algebraic {$K$}-theory}, Lecture
Notes in Mathematics, vol. 1408, Springer-Verlag, Berlin, 1989, Mathematica
Gottingensis.
%
\bibitem{Molinier-Twisted}
R.~Molinier, \emph{Cohomology with twisted coefficients of classifying spaces of a fusion system},
Topology and its Applications \textbf{212} (2016), 1--18.
%
\bibitem{Molinier-Solvable}
R.~Molinier, \emph{Cohomology of linking systems with twisted coefficients by a $p$-solvable actions},
Homology, Homotopy and Applications, \textbf{19} (2017), 61--82.
%
\bibitem{OliverVentura}
B.~Oliver and J.~Ventura, \emph{Extensions of linking systems with $p$-group kernel}, 
Math. Ann. \textbf{338} (2007), 983-1043.
%
\bibitem{Penner-Book} R.~Penner, \emph{Topology and K-Theory, Lecture Notes by Quillen}, Lecture Notes in Mathematics \textbf{2262}, Springer, 2020.
%
\bibitem{Quillen-KTheory} 
D.~Quillen, \emph{Higher algebraic K-theory: I. In: Bass, H. (eds) Higher K-Theories}, Lecture Notes in Mathematics \textbf{341}, 
Springer, Berlin, Heidelberg, 1973.

\bibitem{Richter-Book} 
B.~Richter, \emph{From Categories to Homotopy Theory}, Cambridge Studies in Advanced Mathematics 188,
Cambridge University Press, Cambridge 2020.
%
\bibitem{Schubert-Book}
H.~Schubert, \emph{Categories}, 1972, Springer-Verlag.
%
\bibitem{Slominska-Homotopy} J. S\l omi\' nska, Homotopy colimits on EI-categories, in Algebraic topology (Pozna\' n 1989), 
Lecture Notes in Math. \textbf{1474} (1991), 273--294, Springer-Verlag, New York..
%
\bibitem{Vespa-Survey}
C.~Vespa, \emph{Functor Homology: Theory and Applications}, https://arxiv.org/abs/2111.09719.
%
\bibitem{Weibel-Book}
C.~A.~Weibel, \emph{An Introduction to Homological Algebra}, Cambridge Studies in Advanced Mathematics {\bf 38}, 
Cambridge University Press, 1994. 
%
\bibitem{Webb-Survey} 
P.~Webb, \emph{An introduction to the representations and cohomology of categories}, 
pp. 149-173 in: M. Geck, D. Testerman and J. Thvenaz (eds.), 
Group Representation Theory, EPFL Press (Lausanne) 2007.

%
\bibitem{Webb-Bisets}
P.~Webb, \emph{Biset functors for categories}, preprint, 2023 (https://arxiv.org/abs/2304.06863).
%
\bibitem{Xu-CohSmall}
F.~Xu, \emph{On the cohomology rings of small categories}, J. Pure Appl. Algebra {\bf 212} (2008), 2555-2569.
%
\bibitem{Xu-OnLocal}
F.~Xu, \emph{On local categories of finite groups}, Math Z. {\bf 272} (2012), 1023--1036.
%
\bibitem{Yalcin-Sharp}
E.~Yal\c c\i n, \emph{Higher limits over the fusion orbit category}, Adv. Math. 406 (2022) 108482.
%

\end{thebibliography}
\end{document}